\documentclass{tran-l}
\usepackage{amsfonts}
\usepackage{color}
\usepackage{soul}
\begin{document}
%\usepackage[inactive]{srcltx} % SRC Specials for DVI Searching

% Over-full v-boxes on even pages are due to the \v{c} in author's name
\vfuzz2pt % Don't report over-full v-boxes if over-edge is small
\newcommand{\red}{\color{red}}
\newcommand{\blue}{\color{blue}}
\newcommand{\black}{\color{black}}

% THEOREM Environments ---------------------------------------------------
 \newtheorem{thm}{Theorem}[section]
 \newtheorem{cor}{Corollary}[section]
 \newtheorem{lem}{Lemma}[section]
 \newtheorem{prop}{Proposition}[section]
 \theoremstyle{definition}
 \newtheorem{defn}{Definition}[section]
 \theoremstyle{remark}
 \newtheorem{rem}{Remark}[section]
 \numberwithin{equation}{section}
% MATH -------------------------------------------------------------------
\newcommand{\CC}{\mathbb{C}}
\newcommand{\KK}{\mathbb{K}}
\newcommand{\ZZ}{\mathbb{Z}}
\newcommand{\RR}{\mathbb{R}}
\def\a{{\alpha}}

\def\b{{\beta}}

\def\d{{\delta}}

\def\g{{\gamma}}

\def\l{{\lambda}}

\def\gg{{\mathfrak g}}
\def\cal{\mathcal}

\title{On naturally graded Lie and (non-Lie) Leibniz superalgebras}

\author{L.M. Camacho}
\address{Luisa Mar\'{i}a Camacho \newline \indent
Dpto. Matem{\'a}tica Aplicada I, Universidad de Sevilla, Sevilla
 (Spain) }
\email{lcamacho@us.es}

\author{R.M. Navarro}
\address{Rosa Mar{\'\i}a Navarro \newline \indent
Dpto. de Matem{\'a}ticas, Universidad de Extremadura, C{\'a}ceres
 (Spain) }
\email{rnavarro@unex.es}
\thanks{This work has been  supported  by Agencia Estatal de Investigaci\'on (Spain), grant MTM2016-79661-P (European FEDER support included, UE), by the PCI of the UCA `Teor\'\i a de Lie y Teor\'\i a de Espacios de Banach' and by the PAI with project number FQM298.}

\author{J.M. S\'{a}nchez}
\address{Jos\'{e} M. S\'{a}nchez \newline \indent
Dpto. de Matem\'aticas, Universidad de C\'adiz, Campus de Puerto Real, C\'adiz (Spain)}
\email{{\tt txema.sanchez@uca.es}}

\begin{abstract}
In general, the study of gradations has always represented a cornerstone in algebra theory. In particular, \textit{naturally graded} seems to be the first and the most relevant gradation when it comes to nilpotent algebras,  a large class of solvable ones. In fact, many families of relevant solvable algebras have been obtained by extensions of naturally graded nilpotent algebras, i.e. solvable algebras with a well-structured nilradical. Thus, the aim of this work is introducing the concept of naturally graded for superalgebra structures such as Lie and (non-Lie) Leibniz. After having defined naturally graded Lie and  Leibniz superalgebras, we   characterize natural gradations on a very important class of each of them, that is, those  with maximal super-nilindex.

\end{abstract}

\maketitle


%%%%%%%%%%%%%%%%%%%%%%%%%%%%%%%%%%%%%%%%%%%%%%%%%%%%%%%%%%%%%%%
%%%%%%%%%%%%%%%%%%%%%%%%%%%%%%%%%%%%%%%%%%%%%%%%%%%%%%%%%%%%%%%
\section{Introduction}
%%%%%%%%%%%%%%%%%%%%%%%%%%%%%%%%%%%%%%%%%%%%%%%%%%%%%%%%%%%%%%%
%%%%%%%%%%%%%%%%%%%%%%%%%%%%%%%%%%%%%%%%%%%%%%%%%%%%%%%%%%%%%%%

In general, the study of graded algebras  have always played a fundamental role into Lie theory. Moreover, in nilpotent Lie/Leibniz algebras among all the possible gradations to be considered, natural gradation seems to be the most important due to the fact that it is defined through the filtration which derives from the descending central sequence (see \cite{1,3}). Recently,  the importance of naturally graded Lie/Leibniz algebras has been increased by means of the use of them as nilradical of relevant families of solvable ones  (see  \cite{ JAlgebra,Ancochea, 2,2c}). Therefore, it remains as an inmediate and future work the use of the results of the present paper to obtain important families of solvable Lie/Leibniz superalgebras by extensions of non-nilpotent outer derivations, for instance. Recall that nowadays the classification of solvable algebras is an open and unsolved problem.

Filiform Lie algebras, on one hand, were firstly introduced in \cite{Vergne} by Vergne and they have important properties such as every filiform Lie algebra can be obtained by a deformation of the model filiform algebra $L_n$. Furthermore this type of nilpotent Lie algebras is considered to be the least nilpotent, that is, with the longest possible descending central sequence and maximal nilindex.  The generalization of filiform  Lie algebras for Lie superalgebras has already been obtained  (see \cite{Bor07})  and, in the same way as occurs for Lie algebras, all filiform Lie superalgebras can be obtained by infinitesimal deformations of the model Lie superalgebra $L^{n,m}$ and have maximal super-nilindex. Thus,  we begin the present work introducing the concept of naturally graded for Lie/Leibniz superalgebras and secondly we tackle the problem of determining which filiform Lie superalgebras are naturally graded.  For approaching this second problem we will consider some special infinitesimal deformations of $L^{n,m}$ which are defined by even 2-cocycles in $Z^2_0(L^{n,m},L^{n,m})$.  Throughout this work we give a complete classification (up to isomorphism) of naturally graded filiform Lie superalgebras of dimension less or equal to $7$ as well as some other classifications with either $n$  or $m$ arbitrary.

On the other hand, the notion of Leibniz superalgebras as a generalization of Leibniz algebras was firstly introduced in
\cite{AO0},  general graded Leibniz algebras  was considered before in work \cite{Liv}, though.  Since Leibniz algebras are a generalization of Lie algebras \cite{Loday}, consequently  many of the features of Leibniz superalgebras are generalization of Lie superalgebras \cite{AOK, CGOK1, CGOK2, GOK}.  Likewise,   the study of nilpotent Leibniz algebras \cite{AO0, AO1, AO2} can be very useful to study nilpotent Leibniz superalgebras. The latter constitute the context of the last part of our work, nilpotent Leibniz superalgebras. Thus, as naturally graded Leibniz algebras have been already obtained in \cite{AO1}, our main purpouse throughout the last part of this paper is  obtaining the analog for the case of superalgebras. More precisely, we characterize naturally graded Leibniz superalgebras  with maximal super-nilindex. Note that the techniques used for Leibniz superalgebras are totally different from the ones used for Lie superalgebras and involve a huge amount of computation. Remark that {\em Mathematica} software has been a truly  useful tool to successfully accomplish this task.

Throughout the present paper we will consider vector spaces and algebras over the field of complex numbers  $\mathbb{C}$. 

%%%%%%%%%%%%%%%%%%%%%%%%%%%%%%%%%%%%%%%%%%%%%%%%%%%%%%
%%%%%%%%%%%%%%%%%%%%%%%%%%%%%%%%%%%%%%%%%%%%%%%%%%%%%%
\section{Preliminary results}
%%%%%%%%%%%%%%%%%%%%%%%%%%%%%%%%%%%%%%%%%%%%%%%%%%%%%%
%%%%%%%%%%%%%%%%%%%%%%%%%%%%%%%%%%%%%%%%%%%%%%%%%%%%%%

\subsection{Preliminary results for Lie superalgebras}
 A vector space $V$ is said to be $\ZZ_2${\em -graded} if it admits a decomposition in direct sum, $V=V_{\bar 0} \oplus V_{\bar 1}$, where ${\bar 0}, {\bar 1} \in \ZZ_2$, and such that $V_{\bar i}V_{\bar j} \subset V_{\overline{i+j}}$ for any ${\bar i},{\bar j} \in \ZZ_2$. An element $X \in V$ is called {\it homogeneous of degree} $\bar{i}$ if it is an element of $V_{\bar i}, {\bar j} \in \ZZ_2$. In particular, the elements of $V_{\bar 0}$ (resp. $V_{\bar 1}$) are also called {\em even} (resp. {\em odd}). A {\em Lie superalgebra} (see \cite{Fuks, Scheunert}) is a $\ZZ_2$-graded vector space  $\gg=\gg_{\bar 0} \oplus \gg_{\bar 1}$, with an even bilinear commutation operation (or ``supercommutation'') $[\cdot,\cdot]$, which satisfies the conditions
\begin{enumerate}
\item[1.] $[X,Y]=-(-1)^{\bar{i}\bar{j}}[Y,X],$

\item[2.]  $(-1)^{\bar{k}\bar{i}}[X,[Y,Z]] + (-1)^{\bar{i}\bar{j}}[Y,[Z,X]] + (-1)^{\bar{j}\bar{k}}[Z,[X,Y]]=0$ {\it (super Jacobi identity)}
\end{enumerate}
\indent for all $X \in \gg_{\bar{i}}, Y \in \gg_{\bar{j}}, Z \in \gg_{\bar{k}}$ with $\bar{i},\bar{j},\bar{k} \in \ZZ_2$.

Thus, $\gg_{\bar 0}$ is an ordinary Lie algebra, and $\gg_{\bar 1}$ is a module over $\gg_{\bar 0}$; the Lie superalgebra structure also contains the symmetric pairing $S^2 \gg_{\bar 1} \longrightarrow \gg_{\bar 0}$.  Hereafter, any time we write $[X_i,X_j]$ or $[X_i,Y_j]$ with $X_i$ even and $Y_j$ odd elements, the products $[X_j,X_i]$, $[Y_j,X_i]$ are asumed to be obtained by the skew-symmetric property as it is usual in Lie theory. Also, for simplicity the symmetric products $[Y_i,Y_j]=[Y_j,Y_i]$ are denoted by $(Y_i,Y_j)$. Note that when we deal with Leibniz  superalgebras we do not have skew-symmetric nor symmetric property and each single bracket product must be specified.

In general, the  {\em descending central sequence}  of a Lie superalgebra $\gg=\gg_{\bar 0} \oplus \gg_{\bar 1}$ is defined in the same way as for Lie algebras:  ${\cal C}^0(\gg): =\gg$, ${\cal C}^{k+1}(\gg):=[{\cal C}^k(\gg),\gg]$  for all $k\geq 0$. Consequently, if ${\cal C}^k(\gg)=\{0\}$ for some $k$, then the Lie superalgebra is
called  {\em nilpotent}.  Nevertheless, there are also defined two others descending sequences called ${\cal C}^{k}(\gg_{\bar 0})$ and ${\cal C}^{k}(\gg_{\bar 1})$ which will be  important in our study. They are defined by  ${\cal C}^0(\gg_{\bar i}):=\gg_{\bar i}, {\cal C}^{k+1}(\gg_{\bar i}) := [\gg_{\bar 0}, {\cal C}^k(\gg_{\bar i})], \quad k\geq 0, {\bar i} \in \ZZ_2.$ 

If $\gg=\gg_{\bar 0} \oplus \gg_{\bar 1}$ is a nilpotent Lie superalgebra, then $\gg$ has  {\em super-nilindex} or {\em s-nilindex} $(p,q)$ if satisfies $$({\cal C}^{p-1}(\gg_{\bar 0})) \neq 0, \qquad ({\cal C}^{q-1}(\gg_{\bar 1}))\neq 0, \qquad {\cal C}^{p}(\gg_{\bar 0})={\cal C}^{q}(\gg_{\bar 1})=0.$$ 

It can be noted also that a module $A = A_{\bar 0} \oplus A_{\bar 1}$ of the Lie superalgebra $\mathfrak{g}$ is an even bilinear map $\mathfrak{g} \times A \to A$ satisfying $X(Ya)-(-1)^{\bar{i}\bar{j}}Y(Xa)=[X,Y]a$ for $X \in \mathfrak{g}_{\bar i}$, $Y \in \mathfrak{g}_{\bar j}$, $\bar{i},\bar{j} \in \ZZ_2$, and $a\in A$.

Lie superalgebra cohomology is defined in the  well-known way ( see e.g. \cite{Fuks, cohomology} ).  Thus,  we have the {\it cohomology groups}
$H^q_p(\gg;A):=Z^q_p(\gg;A)\left/ B^q_p(\gg;A) \right.$  where, in particular, the elements of $Z^q_0(\gg;A)$ and $Z^q_1(\gg;A)$ are called {\it even q-cocycles} and {\it odd q-cocycles}, respectively.

 We denote the variety of nilpotent Lie superalgebras $\gg=\gg_{\bar 0} \oplus \gg_{\bar 1}$ with $\dim (\gg_{\bar 0})=n+1$ and $\dim (\gg_{\bar 1})=m$ with the notation ${\cal N}^{n+1,m}$, in complete analogy to Lie algebras (see \cite{ Yu2, Yu3}). Then, any nilpotent Lie superalgebra $\gg=\gg_{\bar 0} \oplus \gg_{\bar 1} \in {\cal N}^{n+1,m}$ with s-nilindex $(n,m)$ is called \textit{filiform} \cite{JGP}.  We denote by ${\cal F}^{n+1,m}$ the subset of ${\cal N}^{n+1,m}$ consisting of all the filiform Lie superalgebras. Likewise, it can be seen that any filiform Lie superalgebra can be always expressed in a suitable basis or so-called {\it adapted basis}. Therefore, if $\gg=\gg_{\bar 0} \oplus \gg_{\bar 1} \in {\cal F}^{n+1,m}$, then there exists an adapted basis of $\gg$, namely $\{X_0, X_1, \dots, X_n, Y_1,\dots,Y_m\}$, with $\{X_0, X_1, \dots, X_n\}$ a basis of $\gg_{\bar 0}$ and $\{Y_1,\dots,Y_m\}$ a basis of $\gg_{\bar 1}$, such
that
$$\begin{array}{lll}
   [X_0,X_i]=X_{i+1},& 1\leq i \leq n-1, &[X_0,X_{n}]=0                   \\[1mm]
   [X_0,Y_j]=Y_{j+1},& 1\leq j\leq m-1, & 
   [X_0,Y_m]=0     
  \end{array}$$

From this result it can be observed that the simplest filiform Lie superalgebra or so-called the {\it model filiform Lie superalgebra}, denoted by $L^{n,m}$, will be defined by the only  non-zero bracket products
$$L^{n,m}:
       \left\{\begin{array}{ll}
          [X_0,X_i]=X_{i+1},& 1\leq i \leq n-1\\[1mm]
          [X_0,Y_j]=Y_{j+1},& 1\leq j \leq m-1
       \end{array}\right.$$

$L^{n,m}$ will be the most important filiform Lie superalgebra, in complete analogy to Lie algebras, since all the other filiform
Lie superalgebras can be obtained from it by deformations. These infinitesimal deformations will be given by the even 2-cocycles, $Z^2_0(L^{n,m},L^{n,m})$, which can be decomposed in the  way 

$$\begin{array}{rl}
Z^2_0(L^{n,m},L^{n,m}) =& \underbrace{Z^2(L^{n,m},L^{n,m})\cap {\rm Hom} (\gg_{\bar 0} \wedge \gg_{\bar 0},\gg_{\bar 0})}_A \oplus \underbrace{Z^2(L^{n,m},L^{n,m})\cap {\rm Hom}
(\gg_{\bar 0} \wedge \gg_{\bar 1},\gg_{\bar 1})}_B \oplus\\
& \underbrace{Z^2(L^{n,m},L^{n,m})\cap {\rm Hom}
(S^2 \gg_{\bar 1},\gg_{\bar 0})}_C \ = \ A \oplus B \oplus C\\
\end{array}$$
 where $\gg_{\bar 0}=L^{n,m}_{\bar 0}$ and $\gg_{\bar 1}=L^{n,m}_{\bar 1}$. The  component $C$  has been determined in \cite{Bor07, JGP2}, and components $A,B$ have been  determined in \cite{JGP3} and \cite{complete}, respectively. Moreover, using these deformations it has been obtained a complete classification (up to isomorphism) of complex filiform Lie superalgebras of dimension less or equal to $7$ (for more details see \cite{lowfiliformLS}). The main result used for this classification is the fact that every filiform Lie superalgebra can be obtained by $L^{n,m}+\Psi$ with $\Psi$ an infinitesimal deformation verifying $\Psi \circ \Psi(x,y,z)=\Psi ( \Psi (x,y),z)+\Psi ( \Psi (z,x),y)+\Psi ( \Psi (y,z),x)=0$.

\subsection{Preliminary results for  Leibniz superalgebras}
Many results and definitions of the above section can be extended for Leibniz superalgebras.

\begin{defn}\label{defB} \cite{AO0}. A $\ZZ_2$-graded vector space $L=L_{\bar 0} \oplus L_{\bar 1}$ is called a  {\em Leibniz superalgebra} if it is equipped with a product $[\cdot,\cdot]$ which satisfies the condition $$[x,[y,z]]=[[x,y],z]-(-1)^{\bar{i}\bar{j}}[[x,z],y] \hspace{0.4cm} \hbox{{\it (super Leibniz identity)}}$$
for all $x \in L$, $y\in L_{\bar{i}}$, $z \in L_{\bar{j}}$, $\bar{i},\bar{j} \in \ZZ_2$. 
\end{defn}

Note that if a Leibniz superalgebra $L$ satisfies the identity $[x,y]=-(-1)^{\bar{i}\bar{j}}[y,x]$ for $x \in L_{\bar{i}}$, $y \in L_{\bar{j}}$, then the super Leibniz identity  becomes the super Jacobi identity.  Therefore Leibniz superalgebras are a generalization of Lie superalgebras. In the same way as for Lie superalgebras, isomorphisms are assumed to be consistent with the $\ZZ_2$-graduation.

If we denote by $R_x$ the right multiplication operator, i.e. $R_x : L \rightarrow L$  given as $R_x(y) := [y,x]$ for $y \in L$, then the super Leibniz identity can be expressed as
$R_{[x,y]} = R_yR_x-(-1)^{\bar{i}\bar{j}}R_xR_y$
where $x\in L_{\bar i}, y \in L_{\bar j}$. 

We denote by $R(L)$ the set of all right multiplication operators. It is not difficult to prove that $R(L)$ with the multiplication 
\begin{equation}\label{product_R(L)}
<R_a,R_b>:=R_aR_b-(-1)^{\bar{i}\bar{j}}R_bR_a
\end{equation}
for $R_a\in {R(L)_{\bar{i}}}$, $R_b\in {R(L)_{\bar{j}}}$, becomes a Lie superalgebra. 

Note that the  concepts of descending central sequence, the variety of Leibniz superalgebras and Engel's theorem are natural extensions from Lie theory.

Let $V = V_{\bar 0} \oplus V_{\bar 1}$ be the underlying vector space of $L$, $L = L_{\bar 0} \oplus L_{\bar 1} \in Leib^{n,m}$, being $Leib^{n,m}$ the variety of Leibniz superalgebras,  and let $G(V)$ be the group of the invertible linear mappings of the form  $f=f_{\bar 0}+f_{\bar 1}$, such that $f_{\bar 0} \in GL(n,\mathbb{C})$ and $f_{\bar 1} \in GL(m,\mathbb{C})$ (then $G(V) = GL(n,\mathbb{C})\oplus GL(m,\mathbb{C}))$. The action of $G(V)$ on $Leib^{n,m}$ induces an action
on the Leibniz superalgebras variety: two laws  $\lambda_1, \lambda_2$ are {\em isomorphic}  if there exists a linear mapping $f = f_{\bar 0}+f_{\bar 1} \in G(V)$, such that  $$\lambda_2(x,y) = f_{\bar{i}+\bar{j}}^{-1}(\lambda_1(f_{\bar{i}}(x),
f_{\bar{j}}(y))), \hspace{0.2cm} \hbox{for any} x \in V_{\bar{i}}, y \in V_{\bar{j}}.$$

The description of the variety of any class of algebras or superalgebras is a difficult problem. Different papers (for example, \cite{AOR, BS, libroKluwer, GO}) are concerning the applications of algebraic groups theory to the description of the variety of Lie/Leibniz algebras.

\begin{defn}
 For a Leibniz superalgebra $L=L_{\bar 0} \oplus L_{\bar 1}$ we define the {\em right annihilator of} $L$ as the set $Ann(L):=\{x \in L : [L,x]=0\}$. 
\end{defn}

It is easy to see that $Ann(L)$ is a two-sided ideal of $L$ and $[x,x] \in Ann(L)$ for any $x \in L_{\bar 0}$. This notion is good and compatible with the right annihilator in Leibniz algebras. If we consider  the ideal generated as  $I:=ideal<[x,y]+(-1)^{\bar{i}\bar{j}}[y,x] : x \in L_{\bar{i}}, y \in L_{\bar{j}}>$, then $I \subset Ann(L).$

Let $L=L_{\bar 0} \oplus L_{\bar 1}$ be a nilpotent Leibniz superalgebra with $\dim L_{\bar 0}=n$ and $\dim L_{\bar 1}=m.$ From Equation \eqref{product_R(L)} we have that $R(L)$ is a Lie superalgebra, and in particular $R(L_{\bar 0})$ is a Lie algebra. As $L_{\bar 1}$ has $L_{\bar 0}$-module structure we can consider $R(L_{\bar 0})$ as a subset of $GL(V_{\bar 1})$  , where  $V_{\bar 1}$  is the underlying vector space of $L_{\bar 1}$. So, we have a Lie algebra formed by nilpotent endomorphisms of $V_{\bar 1}$. Applying the Engel's theorem \cite{Jac} we have the existence of a sequence of subspaces of $V_{\bar 1}$, 
$V_0 \subset V_1 \subset V_2 \subset \dots \subset V_m = V_{\bar 1},$ with $R(L_{\bar 0})(V_{\overline{i+1}})\subset V_{\bar i}.$ Then, it can defined the descending sequences $C^k(L_{\bar 0})$ and $C^k(L_{\bar 1})$ and 
 the  super-nilindex  in the same way as for Lie superalgebras.

Similar to the case of  null-filiform  Leibniz algebras, it is easy to check that a Leibniz superalgebra is  null-filiform  if and only if it is single-generated. Moreover, a  null-filiform  superalgebra  has the maximal super-nilindex (see \cite{null-filiform}).

%%%%%%%%%%%%%%%%%%%%%%%%%%%%%%%%%%%%%%%%%%%%%%%%%%%%%%
%%%%%%%%%%%%%%%%%%%%%%%%%%%%%%%%%%%%%%%%%%%%%%%%%%%%%%
\section{Naturally graded  Lie/Leibniz superalgebras}
%%%%%%%%%%%%%%%%%%%%%%%%%%%%%%%%%%%%%%%%%%%%%%%%%%%%%%
%%%%%%%%%%%%%%%%%%%%%%%%%%%%%%%%%%%%%%%%%%%%%%%%%%%%%%

In terms of algebras, recall that for nilpotent Lie/Leibniz algebras $\gg$, the descending central sequence defines a filtration over the algebra and  therefore there is   associated a graded Lie/Leibniz algebra structure  that we denote as $$gr \gg := \black \sum {\cal C}^{i-1}(\gg)/{\cal C}^{i}(\gg)$$
 If $\gg, gr \gg$ are isomorphic, then $\gg$ is said to be {\bf {\em naturally graded}}.  In particular, there are only two non-isomorphic naturally graded filiform Lie algebras called $L_n$ and $Q_n$, whose laws can be expressed in an adapted basis  $\{X_0, X_1, \dots, X_n\}$ by the following non-zero bracket products (for more details it can be consulted  \cite{libroKluwer, Vergne})
$$\begin{array}{ll}
L_n:  \left\{\begin{array}{ll}
       [X_0,X_i]=X_{i+1}, & \ 1 \leq i \leq n-1\\[1mm]
        \end{array}\right.&
Q_n:  \left\{\begin{array}{ll}
       [X_0,X_i]=X_{i+1}, & \ 1 \leq i \leq n-1\\[1mm]
       [X_i,X_{n-i}]=(-1)^{i}X_n, & \ 1 \leq i \leq n-1\\[1mm]
        \end{array}\right.
    \end{array}$$

Note that both algebras are $(n+1)$-dimensional, although $Q_n$ only appears in the cases $n$ odd and $n\geq 5$.

Analogously, for superalgebra structures, i.e. nilpotent Lie/Leibniz superalgebras  $\gg=\gg_{\bar 0} \oplus \gg_{\bar 1}$,  can be considered the descending sequences ${\cal C}^{k}(\gg_{\bar i}) := \gg_{\bar i}$, ${\cal C}^{k+1}(\gg_{\bar i}) := [\gg_{\bar 0}, {\cal C}^k(\gg_{\bar i})]$, with $k \geq 0$, $\bar{i} \in \ZZ_2$. 

\noindent It can be seen that these sequences define a filtration over $\gg_{\bar 0}$ and $\gg_{\bar 1}$, respectively. Thus, we have on one hand a structure of graded Lie/Leibniz algebra  $$gr \gg_{\bar 0} = \sum {\cal C}^{i-1}(\gg_{\bar 0})/{\cal C}^i(\gg_{\bar 0})$$ and on the other hand a structure of graded $\gg_{\bar 0}$-module $$gr \gg_{\bar 1} = \sum {\cal C}^{i-1}(\gg_{\bar 1})/{\cal C}^{i}(\gg_{\bar 1})$$ 
If we denote $\gg^{i}_{\bar 0}:={\cal C}^{i-1}(\gg_{\bar 0})/{\cal C}^{i}(\gg_{\bar 0})$ and $\gg^{i}_{\bar 1}:={\cal C}^{i-1}(\gg_{\bar 1})/{\cal C}^{i}(\gg_{\bar 1})$, then it is verified that $$[\gg^{i}_{\bar 0}, \gg^{j}_{\bar 0}] \subset \gg^{i+j}_{\bar 0} \qquad \mbox{  and  } \qquad [\gg^{i}_{\bar 0}, \gg^{j}_{\bar 1}] \subset \gg^{i+j}_{\bar 1}$$

\begin{defn}
Given a nilpotent Lie (resp. Leibniz) superalgebra $\gg = \gg_{\bar 0} \oplus \gg_{\bar 1}$, consider $\gg^{i}=\gg^{i}_{\bar 0} \oplus \gg^{i}_{\bar 1}$, with $\gg^{i}_{\bar 0}={\cal C}^{i-1}(\gg_{\bar 0})/{\cal C}^{i}(\gg_{\bar 0})
 \mbox{  and  }  \gg^{i}_{\bar 1}={\cal C}^{i-1}(\gg_{\bar 1})/{\cal C}^{i}(\gg_{\bar 1})$. Thus, $\gg$ is said to be {\bf {\em naturally graded}} if the following conditions hold:
\begin{itemize}
\item[$1.$]  $gr (\gg)=\sum_{i \in \mathbb{N}} \gg^{i}$  is a graded Lie (resp. Leibniz) superalgebra ($[\gg^{i},\gg^{j}]\subset \gg^{i+j}$),
\item[$2.$] $\gg, gr(\gg)$ are isomorphic.
\end{itemize}
\end{defn}

\begin{rem}
The condition $1$ is in fact required in the above definition. Note that in contrast with Lie/Leibniz algebras, $gr (\gg)$ could not be graded. Indeed, if we consider the nilpotent  Leibniz superalgebra (which is also a Lie superalgebra) expressed in the adapted basis $\{X_1, X_2, X_3, X_4, Y_1\}$ as $$\left\{\begin{array}{lll}
[X_2,X_1]=X_3, &  [X_3,X_1]=X_4, & [X_1,X_2]=-X_{3}\\[1mm]
[X_1,X_2]=-X_3, &  [X_1,X_3]=-X_4, &  [Y_1,Y_1]=X_4\\[1mm]
\end{array}\right.$$
we have $\gg^1=\gg^1_{\bar 0} \oplus \gg^1_{\bar 1} = ({\cal C}^0(\gg_{\bar 0})/{\cal C}^1(\gg_{\bar 0})) \oplus 
({\cal C}^0(\gg_{\bar 1})/{\cal C}^1(\gg_{\bar 1})) = <X_1,X_2>\oplus <Y_1>$
Analogously,  $L^2=<X_3>$ and $\gg^3=<X_4>$. Nevertheless, as $(Y_1,Y_1)=X_4$ then it is not verified that $[\gg^1,\gg^1]\subset \gg^2$ and consequently it is not graded.             
\end{rem}

%%%%%%%%%%%%%%%%%%%%%%%%%%%%%%%%%%%%%%%%%%%%%%%%%%%%%%%%%%%%%%
%%%%%%%%%%%%%%%%%%%%%%%%%%%%%%%%%%%%%%%%%%%%%%%%%%%%%%%%%%%%%% 
\section{Naturally graded filiform Lie superalgebras}
%%%%%%%%%%%%%%%%%%%%%%%%%%%%%%%%%%%%%%%%%%%%%%%%%%%%%%%%%%%%%%
%%%%%%%%%%%%%%%%%%%%%%%%%%%%%%%%%%%%%%%%%%%%%%%%%%%%%%%%%%%%%%

Among all the nilpotent Lie superalgebras, the filiform ones (those with maximal super-nilindex) constitutes one of the most important types due to its properties and applications.  Thanks to both the theorem of adapted basis and the definition itself of filiform Lie superalgebras it can be seen that  if $\gg = \gg_{\bar 0} \oplus \gg_{\bar 1}$ is a filiform Lie superalgebra, then its graded superalgebra associated $gr (\gg)=\sum_{i \in \mathbb{N}} \gg^{i}$  with $[\gg^i,\gg^j]\subset \gg^{i+j}$, is exactly
$$\begin{array}{ccccccc}
\underbrace{<X_0,X_1,Y_1>} & \oplus & \underbrace{<X_2, Y_2>} & \oplus &\underbrace{ <X_3,Y_3>} & \oplus & \dots  \\
\gg^1 &  &  \gg^2 & & \gg^3 & &  
\end{array}$$
where for simplicity $\gg^1=\gg^1_{\bar 0} \oplus \gg^1_{\bar 1} = <X_0,X_1>\oplus<Y_1>$ has been replaced by $<X_0,X_1,Y_1>$ as there can not be any possible confusion between even elements, called $X_i$, and odd ones, called $Y_j$.  Similarly,  the same replacement has been applied to every $\gg^i$. The last terms of $gr(\gg)$ depend on three possibilities:

\

1. If $n<m$, then 
\begin{flushright}
$\begin{array}{ccccccccc}
\dots & \oplus & \underbrace{<X_n,Y_n>} & \oplus &\underbrace{ <Y_{n+1}>} & \oplus & \dots & \oplus & \underbrace{<Y_m>} \\
& & \gg^n & &  \gg^{n+1}& &&& \gg^m \\
\end{array}$
\end{flushright}

2. If $n=m$, then 
\begin{flushright}
$\begin{array}{ccccc}
\dots & \oplus & \underbrace{<X_{n-1},Y_{n-1}>} & \oplus & \underbrace{<X_n,Y_n>} \\
& & \gg^{n-1}& & \gg^n \\
\end{array}$
\end{flushright}

3. If $n>m$, then 
\begin{flushright}
$\begin{array}{ccccccccc}
\dots & \oplus & \underbrace{<X_m,Y_m>} & \oplus & \underbrace{<X_{m+1}>} & \oplus & \dots & \oplus & \underbrace{<X_n>}\\
& & \gg^m & & \gg^{m+1}& &&& \gg^n \\
\end{array}$
\end{flushright}

\begin{rem}
It can be seen that the even elements of $gr(\gg)$, that is $\sum \gg^i_{\bar 0}$, constitutes a naturally graded filiform Lie algebra. Furthermore, the definition of naturally graded Lie superalgebras is a generalization of  the one of  naturally graded Lie algebras. Thus, if $\gg=\gg_{\bar 0} \oplus \gg_{\bar 1}$ is a naturally graded Lie superalgebra, then $\gg_{\bar 0}$ is a naturally graded Lie algebra. Indeed, $\gg^i_{\bar 0}$ is equivalent to ${\cal C}^{i-1}(\gg_{\bar 0})/{\cal C}^i(\gg_{\bar 0})$. Recall that the latter is used in the definition of naturally graded Lie algebras. 
\end{rem}   

Thanks to this remark we can divide our study into two cases:  either $\gg_{\bar 0}=L_n$ or $\gg_{\bar 0}=Q_n$. Nevertheless, before doing that we are going to obtain all the naturally graded filiform Lie superalgebras for low dimensions. We can accomplish such classification due to the results of \cite{lowfiliformLS}.

\begin{thm}
Let $\gg=\gg_{\bar 0} \oplus \gg_{\bar 1}$ be a non-degenerated filiform Lie superalgebra with $\dim(\gg_{\bar 0})=n+1$ and $\dim(\gg_{\bar 1})=m$. If  $\dim(\gg) \leq 7$  and $\gg$ is naturally graded then the law of $\gg$ will be isomorphic to  one  law of the following list. 

\

\noindent {\bf List of  laws }

\

\noindent {\bf Pair of dimensions:} $n=2$, $m=1$. There is only one naturally graded filiform Lie superalgebra whose law can be expressed in an adapted basis  by
$$L^{2,1}+\varphi_{1,2}:\left\{\begin{array}{ll}
[X_0,X_1]=X_{2}, & (Y_1,Y_1)=X_2
\end{array}\right.$$ 

\noindent {\bf Pair of dimensions:}  $n=2$, $m=2$. There is only one naturally graded filiform Lie superalgebra whose law can be expressed in an adapted basis  as
$$L^{2,2}+\varphi_{1,2}:\left\{\begin{array}{ll}
[X_0,X_1]=X_2, & (Y_1,Y_1)=X_2\\[1mm]
[X_0,Y_1]=Y_2,&
\end{array}\right.$$ 

\noindent {\bf Pair of dimensions:} $n=2$, $m=3$. There is only one naturally graded filiform Lie superalgebra whose law can be expressed in an adapted basis  by
$$L^{2,3}+\varphi_{1,2}:
\left\{\begin{array}{ll}
[X_0,X_1]=X_2, & (Y_1,Y_1)=X_2\\[1mm]
[X_0,Y_i]=Y_{i+1}, & 1 \leq i \leq 2
\end{array}\right.$$ 

\noindent {\bf Pair of dimensions:}  $n=2$, $m=4$. There is only one naturally graded filiform Lie superalgebra whose law can be expressed in an adapted basis by
$$L^{2,4}+\varphi_{1,2}:\left\{\begin{array}{ll}
[X_0,X_1]=X_2, &  [X_0,Y_i]=Y_{i+1}, \ 1 \leq i \leq 3 \\[1mm]
(Y_1,Y_1)=X_2  &  \\[1mm]
\end{array}\right.$$

\noindent {\bf Pair of dimensions:}  $n=3$, $m=2$. There are two non-isomorphic  naturally graded filiform Lie superalgebras whose laws can be expressed in an adapted basis by
$$\begin{array}{ll}
L^{3,2}+\varphi_{1,2}:& L^{3,2}+\Psi_{1,1}^2+\varphi_{1,2}:\\
\left\{\begin{array}{ll}
[X_0,X_i]=X_{i+1}, \ 1 \leq i \leq 2& (Y_1,Y_1)=X_2  \\[1mm]
[X_0,Y_1]=Y_2 &(Y_1,Y_2)=\frac{1}{2}X_3%\\[1mm]
\end{array}\right.&
\left\{\begin{array}{ll}
[X_0,X_i]=X_{i+1}, \ 1 \leq i \leq 2&(Y_1,Y_1)=X_2 \\[1mm] 
[X_0,Y_1]=Y_2 & (Y_1,Y_2)=\frac{1}{2}X_3\\[1mm]
[X_1,Y_1]=Y_2 &
\end{array}\right.
\end{array}$$

\noindent {\bf Pair of dimensions:} $n=4$, $m=2$. There is only one naturally graded filiform Lie superalgebra whose law can be expressed in an adapted basis by
$$L^{4,2}+\overline{\varphi}_{2,4}:\left\{\begin{array}{ll}
[X_0,X_i]=X_{i+1}, \ 1 \leq i \leq 3 &  (Y_1,Y_2)=X_3 \\[1mm]
[X_0,Y_1]=Y_2& (Y_2,Y_2)=X_4\\[1mm]
(Y_1,Y_1)=2X_2  &  
\end{array}\right.$$
                          
\noindent {\bf Pair of dimensions:} $m=3$, $n=3$. There are two non-isomorphic  naturally graded filiform Lie superalgebras whose laws can be expressed in an adapted basis as
$$\begin{array}{ll}
L^{3,3}+\varphi_{1,2}:& L^{3,3}+\Psi_{1,1}^2+\varphi_{1,2}:\\
 \left\{\begin{array}{ll}
[X_0,X_i]=X_{i+1}, \ 1 \leq i \leq 2& (Y_1,Y_1)=X_2 \\[1mm]
[X_0,Y_{j}]=Y_{j+1}, \ 1 \leq j \leq 2 &(Y_1,Y_2)=\frac{1}{2}X_3 
\end{array}\right.&\left\{\begin{array}{ll}
[X_0,X_i]=X_{i+1}, \ 1 \leq i \leq 2 &(Y_1,Y_1)=X_2   \\[1mm] 
[X_0,Y_{j}]=Y_{j+1}, \ 1 \leq j \leq 2& (Y_1,Y_2)=\frac{1}{2}X_3  \\[1mm] 
[X_1,Y_1]=Y_2 & \\[1mm]
[X_1,Y_2]=Y_3 & 
\end{array}\right.
\end{array}$$

\noindent where the usual bracket products $[\ ,\ ]$ are skew-symmetrical and the products deriving from the symmetric pairing $S^2 \gg_{\bar 1} \longrightarrow \gg_{\bar 0}$ and denoted by $( \ , \ ) $ are symmetrical.  
\end{thm}

\begin{rem}
There is no any naturally graded filiform Lie superalgebra for  the cases $n=3, m=1$;  $n=4, m=1$ and $n=5, m=1$.  Note also that all the naturally graded filiform Lie superalgebras of the aforementioned theorem derive from the naturally graded filiform Lie algebra $L_n$. There is no-one deriving from $Q_n$.
\end{rem}

Next, we  start with $\dim(\gg_{\bar 0})=3=n+1$.

%%%%%%%%%%%%%%%%%%%%%%%%%%%%%%%%%%%%%%%%%%%%%%%%%%%%%%%%%%%%%%
%%%%%%%%%%%%%%%%%%%%%%%%%%%%%%%%%%%%%%%%%%%%%%%%%%%%%%%%%%%%%%
\subsection{$n=2$ and $m$ arbitrary}
%%%%%%%%%%%%%%%%%%%%%%%%%%%%%%%%%%%%%%%%%%%%%%%%%%%%%%%%%%%%%%
%%%%%%%%%%%%%%%%%%%%%%%%%%%%%%%%%%%%%%%%%%%%%%%%%%%%%%%%%%%%%%

Throughout this subsection we are going to study naturally graded filiform Lie superalgebras $\gg=\gg_{\bar 0} \oplus \gg_{\bar 1}$ with $\dim(\gg_{\bar 0})=n+1=3$ and $m \geq 5$ since the cases $1 \leq m \leq 4$ have already been studied. 

\begin{thm}
Let $\gg=\gg_{\bar 0} \oplus\gg_{\bar 1}$ be a non-degenerated filiform Lie superalgebra with $\dim(\gg_{\bar 0})=3$ and $\dim(\gg_{\bar 1})=m$ with $m \geq 5$. If $\gg$ is naturally graded then the law of $\gg$  is  isomorphic to 
$$L^{2,m}+\varphi_{1,2}: \left\{\begin{array}{ll}
[X_0,X_1]=X_2 & (Y_1,Y_1)=X_2 \\[1mm]
[X_0,Y_i]=Y_{i+1}, & 1 \leq i \leq m-1   \\[1mm]
\end{array}\right.$$ 
\end{thm}

\begin{proof}
It can be seen that $\gg_{\bar 0}=L_2$ since $\dim(\gg_{\bar 0})=3$ and $\gg_{\bar 0}$ is a naturally graded filiform Lie algebra. Therefore $gr(\gg)$ is exactly 
$$\begin{array}{ccccccccc}
\underbrace{<X_0,X_1,Y_1>} & \oplus & \underbrace{<X_2, Y_2>} & \oplus & \underbrace{<Y_3>}& \oplus & \dots &\oplus & \underbrace{<Y_m>}\\
\gg^{1} & & \gg^{2} & & \gg^{3} & &  && \gg^{m}
 \end{array}$$
 
Thus, the bracket products can be considered to be 
$$\left\{\begin{array}{ll}
[X_0,X_1]=X_{2},&   [X_1,Y_i]=b_{1i}Y_{i+1}, \ 1\leq i \leq m-1\\[1mm]
[X_0,Y_{i}]=Y_{i+1}, \ 1 \leq i \leq m-1,&  [X_2,Y_i]=b_{2i}Y_{i+2}, \ 1 \leq i \leq m-2   \\[1mm]
(Y_1,Y_1)=c_{11}X_2  &  
\end{array}\right.$$
with $c_{11} \neq 0$ in order not to have a degenerated Lie superalgebra, that is, a Lie algebra. By means of a simple change of scale it can be seen that there is no loss of generality in considering $c_{11}=1$. It can be seen that $(Y_1,Y_1)=X_2$ corresponds with the cocycle $\varphi_{1,2}$. Additionally, the remaining structure constants  $b_{1i}, b_{2i}$  indeed will be determined by the cocycles  $\Psi_{1,1}^2, \Psi_{2,1}^3,$ respectively (more details it can be consulted \cite{complete}) given as
$$\begin{array}{lll}
\Psi_{1,1}^2: \ [X_1,Y_j]=Y_{j+1}, \ 1 \leq j \leq m-1 & \quad & \Psi_{2,1}^3:\left\{\begin{array}{l}
[X_1,Y_j]=-(j-1)Y_{j+1}, \ 2 \leq j \leq m-1 \\[1mm]
[X_2,Y_j]=Y_{j+2}, \ 1 \leq j \leq m-2 \\[1mm]
\end{array}\right.
\end{array}$$ 

These cocycles, or so-called  {\em infinitesimal deformations},  are the only one ``naturally graded", that is, of weight $0$. Recall that the {\it weight} of $\Psi_{k,1}^s$ is $s-k-1$. 

Next, we  apply the method of classification used in \cite{lowfiliformLS}: {\it every filiform Lie superalgebra can be expressed by $L^{n,m}+\Psi$,  with  $L^{n,m}$ the law of the model filiform Lie superalgebra and $\Psi$ a linear (infinitesimal) deformation that verifies $\Psi \circ \Psi=0$, being}  $\Psi \circ \Psi(x,y,z)=\Psi ( \Psi (x,y),z)+\Psi ( \Psi (z,x),y)+\Psi ( \Psi (y,z),x)$. Thus, we can consider 
$$L^{2,m}+\Psi: \left\{\begin{array}{ll}
[X_0,X_1] = X_2,  &  [X_1,Y_j]=[a-b(j-1)]Y_{j+1}, \ 2\leq i \leq m-1 \\[1mm]
[X_0,Y_i] = Y_{i+1}, \ 1 \leq i \leq m-1,  &  [X_2,Y_j]=bY_{j+2}, \ 1 \leq i \leq m-2 \\[1mm]
[X_1,Y_1]=aY_2,  &  (Y_1,Y_1)=X_2 \\[1mm]
\end{array}\right.$$
where $\Psi=\varphi_{1,2}+a \Psi_{1,1}^2+ b \Psi_{2,1}^3$. From the condition $\Psi \circ \Psi(X_1,X_2,Y_1)=0$ it is obtained that $b=0$. Then applying the isomorphism, or so-called change of basis, defined by $X'_0=X_0$, $X'_1=X_1-aX_0$, $X'_2=X_2$ and $Y'_j=Y_j$  for any $j \in \mathbb{N}$, it can seen that there is no loss of generality in supposing $a=0$. Consequently, it remains only $L^{2,m} + \varphi_{1,2}$.
\end{proof}

%%%%%%%%%%%%%%%%%%%%%%%%%%%%%%%%%%%%%%%%%%%%%%%%%%%%%%%%%%%%%%
%%%%%%%%%%%%%%%%%%%%%%%%%%%%%%%%%%%%%%%%%%%%%%%%%%%%%%%%%%%%%%
\subsection{$n=3$ and $m=4$} In this case we have the following results:
%%%%%%%%%%%%%%%%%%%%%%%%%%%%%%%%%%%%%%%%%%%%%%%%%%%%%%%%%%%%%%
%%%%%%%%%%%%%%%%%%%%%%%%%%%%%%%%%%%%%%%%%%%%%%%%%%%%%%%%%%%%%%

\begin{thm}
Let $\gg=\gg_{\bar 0} \oplus \gg_{\bar 1}$ be a non-degenerated filiform Lie superalgebra with $\dim(\gg_{\bar 0}) = \dim(\gg_{\bar 1})=4$. If $\gg$ is naturally graded then the law of $\gg$  is  isomorphic to one of the following list of pairwise non-isomorphic laws 

$$\begin{array}{lll}
\begin{array}{l}
L^{3,4}+\varphi_{1,2}:\\ \\
\left\{\begin{array}{l}
   [X_0,X_i]=X_{i+1}\\[1mm]
   [X_0,Y_{j}]=Y_{j+1}\\[1mm] 
   (Y_1,Y_1)=X_2\\[1mm]
   (Y_1,Y_2)=\frac{1}{2}X_3 
   \end{array}\right.
   \end{array}
&
\begin{array}{ll}
L^{3,4}+\varphi_{1,2} + \Psi_{3,1}^4:\\ \\
 \left\{\begin{array}{ll}
   [X_0,X_i]=X_{i+1}&(Y_1,Y_1)=X_2 \\[1mm]
   [X_0,Y_{j}]=Y_{j+1}  & (Y_1,Y_2)=\frac{1}{2}X_3   \\[1mm]
   [X_1,Y_3]=Y_4& \\[1mm]
   [X_2,Y_2]=-Y_4 &\\[1mm]
   [X_3,Y_1]=Y_4&
   
 \end{array}\right.     
                        \end{array}                 
                  &
                  \begin{array}{l}
                  L^{3,4}+\varphi_{1,2} + \Psi_{1,1}^2:\\ \\
 \left\{\begin{array}{l}
   [X_0,X_i]=X_{i+1}  \\[1mm]
   [X_0,Y_{j}]=Y_{j+1}   \\[1mm]
   [X_1,Y_j]=Y_{j+1} \\[1mm]
   (Y_1,Y_1)=X_2\\[1mm]
   (Y_1,Y_2)=\frac{1}{2}X_3
\end{array}\right.
\end{array} \end{array}
$$
with $1 \leq i \leq 2 $ and $1 \leq j \leq 3.$
\end{thm}

\begin{proof}
We consider naturally graded filiform Lie superalgebras $\gg=\gg_{\bar 0} \oplus \gg_{\bar 1}$ with $\dim(\gg_{\bar 0})=4$ and $\dim(\gg_{\bar 1}) = 4.$ Therefore $\gg_{\bar 0}=L_3$ and $gr(\gg)$ is

$$\begin{array}{ccccccccc}
\underbrace{<X_0,X_1,Y_1>} & \oplus & \underbrace{<X_2, Y_2>} & \oplus & \underbrace{<X_3,Y_3>} & \oplus & \underbrace{<Y_4>}\\
\gg^1 & & \gg^2 & & \gg^3 & & \gg^4 &&
\end{array}$$

Thus, the bracket products can be considered to be 
$$\begin{array}{ll}
 \left\{\begin{array}{l}
   [X_0,X_i]=X_{i+1}, \ i = 1,2 \\[1mm]
   [X_0,Y_i]=Y_{i+1}, \ 1 \leq i \leq 3   \\[1mm]
   [X_1,Y_i]=b_{1i}Y_{i+1}, \ 1 \leq i \leq 3  \\[1mm]
   [X_2,Y_i]=b_{2i}Y_{i+2}, \ 1 \leq i \leq 2 \\[1mm]
\end{array}\right.&\begin{array}{ll}
   [X_3,Y_1]=b_{31}Y_4, \\[1mm]
   (Y_1,Y_1)=c_{11}X_2    \\[1mm]
   (Y_1,Y_2)=c_{12}X_3    \\[1mm]
\end{array} \end{array}$$

\noindent In a similar way to the previous case the structure constants can be considered to be determined by the cocycles $\Psi_{1,1}^2, \Psi_{2,1}^3$, $\Psi_{3,1}^4$ and $\varphi_{1,2}$ (see \cite{Bor07, complete}). In general, $\Psi_{1,1}^2$ and $\Psi_{2,1}^3$ have already been described in the above case, and $\Psi_{3,1}^4$ and $\varphi_{1,2}$ are as follows

$$\begin{array}{ll}
\begin{array}{l}
\Psi_{3,1}^4:\left\{\begin{array}{l}
[X_1,Y_3]=Y_{4} \\[1mm]
[X_2,Y_2]= -Y_4\\[1mm]
[X_3,Y_1]=Y_4 \\[1mm]
\end{array}\right.
\end{array}& \begin{array}{l}
\varphi_{1,2}:\left\{\begin{array}{l}
(Y_1,Y_1)=X_2 \\[1mm]
(Y_1,Y_2)=\frac{1}{2}X_3\\[1mm]
\end{array}\right.
\end{array} 
\end{array} 
$$

\noindent Thus, in order not to have a degenerated case, that is, a Lie algebra, the coefficient of 
$\varphi_{1,2}$ must be different from zero and then, it can be easily seen that there is no loss of generality in considering this coefficient equals  $1$.  Therefore we can consider
$$\begin{array}{ll}
&\\
L^{3,4}+\Psi: \left\{\begin{array}{l}
   [X_0,X_1]=X_2,  \\[1mm]
   [X_0,X_2]=X_3,  \\[1mm]
   [X_0,Y_{i}]=Y_{i+1}, \ 1 \leq i \leq 3,  \\[1mm]
   [X_1,Y_1]=a_1Y_2,\\[1mm]
   [X_1,Y_2]=(a_1 - a_2)Y_3, \\[1mm]
   [X_1,Y_3]=(a_1 - 2a_2 + a_3)Y_4, \\[1mm]
\end{array}\right.&\begin{array}{ll}
   [X_2,Y_1]=a_2Y_3 \\[1mm]
   [X_2,Y_2]=(a_2 - a_3)Y_4 \\[1mm]
   [X_3,Y_1]=a_3Y_4 \\[1mm]
   (Y_1,Y_1)=X_2    \\[1mm]
   (Y_1,Y_2)=\frac{1}{2}X_3    \\[1mm]
\end{array} \end{array}$$
being $\Psi=\varphi_{1,2} + a_1 \Psi_{1,1}^2 + a_2 \Psi_{2,1}^3 + a_3 \Psi_{3,1}^4$. From the condition $\Psi \circ \Psi(Y_1,Y_1,Y_1) = 0$ we obtain $a_2=0,$ and on account of $\Psi \circ \Psi(X_1,X_2,Y_1) = 0$ we conclude $a_1a_3 = 0$.
\begin{itemize}
\item If $a_1 = 0$ we distinguish two cases: $a_3 = 0$ or $a_3 \neq 0.$
\begin{itemize}
\item In the case of $a_3 = 0$ we obtain  $L^{3,4}+\varphi_{1,2}$.
\item If $a_3 \neq 0$ without loss of generality we can consider $a_3 = 1$ (by means of a simple change of scale) obtaining $L^{3,4}+\varphi_{1,2} + \Psi_{3,1}^4$.
\end{itemize}
\item If $a_1 \neq 0$ then $a_3 = 0.$ Therefore, applying the isomorphism (change of scale) defined by $X'_0=a_1X_0$, $X'_1=X_1$, $X'_2=a_1X_2$, $X'_3=a_1^2X_3$, $Y'_1= \sqrt{a_1}Y_1,$ $Y'_2=a_1\sqrt{a_1}Y_2$ and $Y'_3=a_1^2\sqrt{a_1}Y_3$, it can seen that there is no loss of generality in supposing $a_1=1,$ and we obtain $L^{3,4}+\varphi_{1,2} + \Psi_{1,1}^2$.
\end{itemize}
\end{proof}

%%%%%%%%%%%%%%%%%%%%%%%%%%%%%%%%%%%%%%%%%%%%%%%%%%%%%%%%%%%%%%
%%%%%%%%%%%%%%%%%%%%%%%%%%%%%%%%%%%%%%%%%%%%%%%%%%%%%%%%%%%%%%
\subsection{$n=3$ and $m \geq 5$}
%%%%%%%%%%%%%%%%%%%%%%%%%%%%%%%%%%%%%%%%%%%%%%%%%%%%%%%%%%%%%%
%%%%%%%%%%%%%%%%%%%%%%%%%%%%%%%%%%%%%%%%%%%%%%%%%%%%%%%%%%%%%%

Throughout this subsection we are going to study naturally graded filiform Lie superalgebras $\gg=\gg_{\bar 0} \oplus \gg_{\bar 1}$ with $\dim(\gg_{\bar 0})=n+1=4$ and $m \geq 5$ since the cases $1 \leq m \leq 4$ have already been studied. 

\begin{thm}
Let $\gg=\gg_{\bar 0} \oplus\gg_{\bar 1}$ be a non-degenerated filiform Lie superalgebra with $\dim(\gg_{\bar 0})=4$ and $\dim(\gg_{\bar 1})=m$ with $m \geq 5$. If $\gg$ is naturally graded then the law of $\gg$  is  isomorphic to one of the two non-isomorphic laws:

$$\begin{array}{ll}
\begin{array}{l}
L^{3,m}+\varphi_{1,2}: \left\{\begin{array}{l}
[X_0,X_i]=X_{i+1}\\[1mm]
[X_0,Y_{j}]=Y_{j+1} \\[1mm]
(Y_1,Y_1)=X_2   \\[1mm]
(Y_1,Y_2)=\frac{1}{2}X_3
   \end{array}\right.
   \end{array}            
                  &
                  \begin{array}{ll}
                  L^{3,m}+\varphi_{1,2} + \Psi_{1,1}^2: \left\{\begin{array}{ll}
   [X_0,X_i]=X_{i+1}&(Y_1,Y_1)=X_2 \\[1mm]
   [X_0,Y_{j}]=Y_{j+1}& (Y_1,Y_2)=\frac{1}{2}X_3  \\[1mm]
   [X_1,Y_j]=Y_{j+1}&
                        \end{array}\right.
                        \end{array}
                        \end{array}$$
\noindent with $1 \leq i \leq 2,$ $1 \leq j \leq m-1.$
\end{thm}

\begin{proof}
Let $\gg = \gg_{\bar 0} \oplus \gg_{\bar 1}$ be a naturally graded filiform Lie superalgebra with $\dim(\gg_{\bar 0})=4$ and $\dim(\gg_{\bar 1}) = m \geq 5.$ Therefore, $\gg_{\bar 0} = L_4$ and $gr(\gg)$ is

\

\noindent $\begin{array}{ccccccccccc}
\underbrace{<X_0,X_1,Y_1>} & \oplus & \underbrace{<X_2, Y_2>} & \oplus & \underbrace{<X_3,Y_3>} & \oplus & \underbrace{<Y_4>} & \oplus & \dots & \oplus & \underbrace{<Y_m>} \\
\gg^{1} & &  \gg^{2} & &  \gg^{3} & & \gg^{4} & & & &  \gg^{m}
\end{array}$

\

So the bracket products are

\

$\begin{array}{lll}
[X_0,X_1]=X_2,  &  & [X_2,Y_i]=b_{2i}Y_{i+2}, \ 1 \leq i \leq m-2 \\[1mm]
[X_0,X_2]=X_3,  &  & [X_3,Y_i]=b_{3i}Y_{i+3}, \ 1 \leq i \leq m-3 \\[1mm]
[X_0,Y_i]=Y_{i+1}, \ 1 \leq i \leq m-1,   & &      (Y_1,Y_1)=c_{11}X_2    \\[1mm]
[X_1,Y_i]=b_{1i}Y_{i+1}, \ 1\leq i \leq m-1,  & &    (Y_1,Y_2)=c_{12}X_3
\end{array}$ 
                        
\

\noindent As in the above case and in order not to have any degenerated case the coefficient of $\varphi_{1,2}$ must be different from zero and then, it can be easily seen that there is no loss of generality in considering this coefficient equals  $1$. Therefore we can consider $L^{3,m} + \Psi,$ with $\Psi=\varphi_{1,2} + a_1 \Psi_{1,1}^2 + a_2 \Psi_{2,1}^3 + a_3 \Psi_{3,1}^4,$ and we have the following products: 

$$L^{3,m}+\Psi: \left\{\begin{array}{l}
[X_0,X_1]=X_2  \\[1mm]
[X_0,X_2]=X_3  \\[1mm]
[X_0,Y_{i}]=Y_{i+1}, \ 1 \leq i \leq m-1   \\[1mm]
[X_1,Y_1]=a_1Y_2,\\[1mm]
[X_1,Y_i]=\bigl(a_1 - (i-1)a_2 + \frac{(i-1)(i-2)}{2}a_3\bigr)Y_{i+1}, \ 2 \leq i \leq m-1 \\[1mm]
[X_2,Y_i]=\bigl(a_2 - (i-1)a_3\bigr)Y_{i+2}, \ 1 \leq i \leq m-2 \\[1mm]
[X_3,Y_i]=a_3Y_{i+3}, \ 1 \leq i \leq m-3 \\[1mm]
(Y_1,Y_1)=X_2    \\[1mm]
(Y_1,Y_2)=\frac{1}{2}X_3    \\[1mm]
\end{array}\right.$$

\noindent  The condition $\Psi \circ \Psi(Y_1,Y_1,Y_1) = 0$ leads to $a_2 = 0$ and then, from the condition $\Psi \circ \Psi(X_1,X_3,Y_1) = 0$ it is obtained that $a_3=0$.  So if $a_1 = 0$ we have $L^{3,m}+\varphi_{1,2}$. Otherwise we have $L^{3,4}+\varphi_{1,2} + \Psi_{1,1}^2$ after applying a simple change of scale.
\end{proof}

%%%%%%%%%%%%%%%%%%%%%%%%%%%%%%%%%%%%%%%%%%%%%%%%%%%%%%%%%%%%%%
%%%%%%%%%%%%%%%%%%%%%%%%%%%%%%%%%%%%%%%%%%%%%%%%%%%%%%%%%%%%%%
\subsection{$n$ and $m$ arbitrary with $n>2m$}
%%%%%%%%%%%%%%%%%%%%%%%%%%%%%%%%%%%%%%%%%%%%%%%%%%%%%%%%%%%%%%
%%%%%%%%%%%%%%%%%%%%%%%%%%%%%%%%%%%%%%%%%%%%%%%%%%%%%%%%%%%%%%

\begin{thm}
There is not any non-degenerated naturally graded filiform Lie superalgebra $\gg = \gg_{\bar 0} \oplus \gg_{\bar 1}$ with the condition $n>2m$, being $\dim(\gg_{\bar 0})=n+1$ and $\dim(\gg_{\bar 1})=m$. 
\end{thm}
\begin{proof}
This result derives from the fact that there is not any cocycle $\varphi_{k,s}$ naturally graded, that is, of weight equals  $0$ since we have always $\max\{\frac{n-2m-1}{2}, n-2m\} \leq $ weight$(\varphi_{k,s})$  (see \cite{Bor07}, Proposition 6.4).   
\end{proof}
\begin{rem}
We have as particular cases of the above theorem  $m=1, n \geq 3$; $m=2, n\geq 5$ and $m=3, n\geq 7$. Consequently,  and in order to study all the posibilities for $m\leq 3,$ only remains to consider the cases $4 \leq n \leq 6$ and $m=3$.
\end{rem} 

%%%%%%%%%%%%%%%%%%%%%%%%%%%%%%%%%%%%%%%%%%%%%%%%%%%%%%%%%%%%%%
%%%%%%%%%%%%%%%%%%%%%%%%%%%%%%%%%%%%%%%%%%%%%%%%%%%%%%%%%%%%%%
\subsection{$n=4$ and $m = 3$}
%%%%%%%%%%%%%%%%%%%%%%%%%%%%%%%%%%%%%%%%%%%%%%%%%%%%%%%%%%%%%%
%%%%%%%%%%%%%%%%%%%%%%%%%%%%%%%%%%%%%%%%%%%%%%%%%%%%%%%%%%%%%%

\begin{thm}
Let $\gg=\gg_{\bar 0} \oplus\gg_{\bar 1}$ be a non-degenerated filiform Lie superalgebra with $\dim(\gg_{\bar 0})=5$ and $\dim(\gg_{\bar 1})=3$. If $\gg$ is naturally graded then the law of $\gg$  is  isomorphic to one of the following non-isomorphic laws:
$$\begin{array}{ll}
\begin{array}{l}
L^{4,3} + \varphi_{2,4}:\\ \\
 \left\{\begin{array}{ll}
[X_0,X_i]=X_{i+1} & (Y_1,Y_3)=-X_4\\[1mm]
[X_0,Y_j]=Y_{j+1} & (Y_2,Y_2)=X_4
\end{array}\right.\end{array}            
                  &
\begin{array}{l}
L^{4,3} + \varphi_{2,4} + \Psi_{1,1}^2 + 2\Psi_{2,1}^3:\\ \\
\left\{\begin{array}{ll}
[X_0,X_i]=X_{i+1}, &  [X_2,Y_1]=Y_3 \\[1mm]
[X_0,Y_j]=Y_{j+1}, &  (Y_1,Y_3)=-X_4 \\[1mm]
[X_1,Y_1]=Y_2  &  (Y_2,Y_2)=X_4 \\[1mm] 
[X_1,Y_2]=-Y_3 & \\[1mm]
\end{array}\right. \end{array}
\end{array}$$
   $$\begin{array}{ll}
\begin{array}{l}
L^{4,3} + \varphi_{1,2} + t\varphi_{2,4}:\\ \\
\left\{\begin{array}{ll}
[X_0,X_i]=X_{i+1},&(Y_1,Y_2)=\frac{1}{2}X_3 \\[1mm]
[X_0,Y_j]=Y_{j+1}, & (Y_1,Y_3)=(\frac{1}{2} -t)X_4 \\[1mm]
(Y_1,Y_1)=X_2, &(Y_2,Y_2)=tX_4
\end{array}\right.\end{array}
   &
\begin{array}{l}
L^{4,3} + 4\varphi_{1,2} + \varphi_{2,4}+ \Psi_{1,1}^2:\\ \\
\left\{\begin{array}{ll}
[X_0,X_i]=X_{i+1},&(Y_1,Y_1)=4X_2  \\[1mm]
[X_0,Y_j]=Y_{j+1},&(Y_1,Y_2)=2X_3   \\[1mm]
[X_1,Y_1]=Y_2,&(Y_1,Y_3)=X_4 \\[1mm]
[X_1,Y_2]=Y_3,&(Y_2,Y_2)=X_4
\end{array}\right.\end{array}
\end{array}$$
with $1 \leq i \leq 3 $, $1 \leq j \leq 2$ and  $t \in \mathbb{C}$  a parameter.
\end{thm}

\begin{proof}
Let $\gg=\gg_{\bar 0} \oplus \gg_{\bar 1}$ be a naturally graded filiform Lie superalgebra with $\dim(\gg_{\bar 0})=5$ and $\dim(\gg_{\bar 1}) = 3.$ Therefore, $\gg_{\bar 0}=L_4$ and $gr(\gg)$ is
$$\begin{array}{ccccccc}
\underbrace{<X_0,X_1,Y_1>} & \oplus & \underbrace{<X_2, Y_2>} & \oplus & \underbrace{<X_3,Y_3>} & \oplus & \underbrace{<X_4>}  \\
\gg^{1} & &  \gg^{2} & &  \gg^{3} & & \gg^{4}
\end{array}$$

In an analogous way to previous cases, we can consider $L^{4,3} + \Psi,$ with $\Psi = c\varphi_{1,2} + d\varphi_{2,4} + a_1\Psi_{1,1}^2 + a_2 \Psi_{2,1}^3,$ and we have the following products:
$$\begin{array}{ll}
L^{3,4}+\Psi: \left\{\begin{array}{l}
[X_0,X_i]=X_{i+1}, \ 1 \leq i \leq 3,  \\[1mm]
[X_0,Y_i]=Y_{i+1}, \ 1 \leq i \leq 2,   \\[1mm]
[X_1,Y_1]=a_1Y_2,\\[1mm]
[X_1,Y_2]=(a_1-a_2)Y_3, \\[1mm]
[X_2,Y_1]=a_2Y_3
\end{array}\right.&\begin{array}{ll}
(Y_1,Y_1)=cX_2    \\[1mm]
(Y_1,Y_2)=\frac{c}{2}X_3    \\[1mm]
(Y_1,Y_3)=(\frac{c}{2} -d)X_4    \\[1mm]
(Y_2,Y_2)=dX_4
\end{array}  \end{array}$$

\noindent where at least $c$ or $d$ is non-zero. The condition $\Psi \circ \Psi(Y_1,Y_1,Y_1) = 0$ give us $a_2c = 0$ and, since $\Psi \circ \Psi(X_1,Y_1,Y_2) = 0,$ we have $a_1d + (a_2-a_1)(\frac{c}{2}-d) = 0$.
\begin{itemize}
\item If $c = 0$ then $d \neq 0$, and from $a_1d + (a_2-a_1)(\frac{c}{2}-d) = 0$ we have $a_2 = 2a_1$.
\begin{itemize}
\item For $a_1 = 0$ is $a_2 = 0$ and $d \neq 0,$ so using a change of scale $d = 1$ and we obtain $L^{4,3} + \varphi_{2,4}$.
\item For $a_1 \neq 0$ again $d \neq 0$, and with a similar change of scale $a_1 = d = 1$ and $a_2 = 2,$ obtaining $L^{4,3} + \varphi_{2,4} + \Psi_{1,1}^2 + 2\Psi_{2,1}^3$.
\end{itemize}

\item In the case of $c \neq 0$ then necessarily $a_2 = 0$ and from $a_1d + (a_2-a_1)(\frac{c}{2}-d) = 0$ is $4a_1d = a_1c$.
\begin{itemize}
\item For $a_1 = 0$  by means of the change of scale $c = 1, d = t,$ being $t$ a parameter,  we  obtain $L^{4,3} + \varphi_{1,2} + t\varphi{2,4}$.
\item For $a_1 \neq 0$ is $d \neq 0$ and with the change of scale $a_1 = d = 1, a_2 = 0$ and $c = 4$  we obtain $L^{4,3} + 4\varphi_{1,2} + \varphi_{2,4} +\Psi_{1,1}^2$.
\end{itemize}
\end{itemize}
\end{proof}

%%%%%%%%%%%%%%%%%%%%%%%%%%%%%%%%%%%%%%%%%%%%%%%%%%%%%%%%%%%%%%
%%%%%%%%%%%%%%%%%%%%%%%%%%%%%%%%%%%%%%%%%%%%%%%%%%%%%%%%%%%%%%
\subsection{$n=5$ and $m = 3$}
%%%%%%%%%%%%%%%%%%%%%%%%%%%%%%%%%%%%%%%%%%%%%%%%%%%%%%%%%%%%%%
%%%%%%%%%%%%%%%%%%%%%%%%%%%%%%%%%%%%%%%%%%%%%%%%%%%%%%%%%%%%%%

\begin{thm}
Let $\gg=\gg_{\bar 0} \oplus \gg_{\bar 1}$ be a non-degenerated filiform Lie superalgebra with $\dim(\gg_{\bar 0})=6$ and $\dim(\gg_{\bar 1})=3$. If $\gg$ is naturally graded then the law of $\gg$ will be isomorphic to the following law

$$\begin{array}{ll}
L^{5,3} + \overline{\varphi}_{2,4}: \left\{\begin{array}{l}
[X_0,X_i]=X_{i+1}, \ 1 \leq i \leq 4, \\[1mm]
[X_0,Y_i]=Y_{i+1}, \ 1 \leq i \leq 2, \\[1mm]
(Y_1,Y_1)=3X_2,    \\[1mm]
(Y_1,Y_2)=\frac{3}{2}X_3,
\end{array}\right.&\begin{array}{ll}
(Y_1,Y_3)=\frac{1}{2}X_4    \\[1mm]
(Y_2,Y_2)=X_4               \\[1mm]
(Y_2,Y_3)=\frac{1}{2}X_5    \\[1mm]
\end{array}\end{array}$$
\end{thm}
\begin{proof}
Let $\gg=\gg_{\bar 0} \oplus \gg_{\bar 1}$ be a naturally graded filiform Lie superalgebra with $dim(\gg_0)=6$ and $dim(\gg_1) = 3.$ Then $gr(\gg)$ is
$$\begin{array}{ccccccccc}
\underbrace{<X_0,X_1,Y_1>} & \oplus & \underbrace{<X_2, Y_2>} & \oplus & \underbrace{<X_3,Y_3>} & \oplus & \underbrace{<X_4>} & \oplus & \underbrace{<X_5>} \\
\gg^{1} & &  \gg^{2} & &  \gg^{3} & & \gg^{4} & & \gg^{5}
\end{array}$$

In this case we have to distinguish two separate sub-cases: $\gg_{\bar 0}=L_5$ or $\gg_{\bar 0}=Q_5$, the two possible naturally graded filiform Lie algebras for this dimension.

\begin{itemize}
\item If $\gg_{\bar 0}=L_5$, and as in previous cases we get $L^{5,3} + \Psi,$ with $\Psi = c\overline{\varphi}_{2,4} + a_1\Psi_{1,1}^2 + a_2\Psi_{2,1}^3,$ being $\overline{\varphi}_{2,4} = 3\varphi_{1,2} + \varphi_{2,4} $ and $c \neq 0$. So we have
$$\begin{array}{ll}
L^{5,3}+\Psi: \left\{\begin{array}{l}
[X_0,X_i]=X_{i+1}, \ 1 \leq i \leq 4, \\[1mm]
[X_0,Y_i]=Y_{i+1}, \ 1 \leq i \leq 2, \\[1mm]
[X_1,Y_1]=a_1Y_2,                     \\[1mm]
[X_1,Y_2]=(a_1-a_2)Y_3,				 \\[1mm]
[X_2,Y_1]=a_2Y_3,
\end{array}\right.&\begin{array}{ll}
(Y_1,Y_1)=3cX_2    \\[1mm]
(Y_1,Y_2)=\frac{3}{2}cX_3 \\[1mm]
(Y_1,Y_3)=\frac{c}{2}X_4    \\[1mm]
(Y_2,Y_2)=cX_4    \\[1mm]
(Y_2,Y_3)=\frac{c}{2}X_5    \\[1mm]
\end{array} \end{array}$$

\noindent The condition $\Psi \circ \Psi(X_1,Y_1,Y_2) = 0$ give us $a_1 = 0$. Since $\Psi \circ \Psi(Y_1,Y_1,Y_1) = 0$ we have $a_2 = 0.$ With a change of scale we can consider $c=1$ and we obtain $L^{5,3} + \overline{\varphi}_{2,4}$.

\item If $\gg_{\bar 0}=Q_5$ we have to add the brackets $[X_1,X_4]=-X_5$ and $[X_2,X_3]=X_5$. Indeed, these brackets stem from cocycles of \cite{JGP3}. All the above reasoning is also valid but from the condition $\Psi \circ \Psi(X_1,Y_2,Y_2) = 0$ we obtain that  $[X_1,X_4]=0$, which is clearly a contradiction. Thus, there is not any non-degenerated filiform Lie superalgebra deriving from the Lie algebra $Q_5$.
\end{itemize}
\end{proof}

%%%%%%%%%%%%%%%%%%%%%%%%%%%%%%%%%%%%%%%%%%%%%%%%%%%%%%%%%%%%%%
%%%%%%%%%%%%%%%%%%%%%%%%%%%%%%%%%%%%%%%%%%%%%%%%%%%%%%%%%%%%%%
\subsection{$n=6$ and $m = 3$}
%%%%%%%%%%%%%%%%%%%%%%%%%%%%%%%%%%%%%%%%%%%%%%%%%%%%%%%%%%%%%%
%%%%%%%%%%%%%%%%%%%%%%%%%%%%%%%%%%%%%%%%%%%%%%%%%%%%%%%%%%%%%%

\begin{thm}
Let $\gg=\gg_{\bar 0} \oplus\gg_{\bar 1}$ be a non-degenerated filiform Lie superalgebra with $\dim(\gg_{\bar 0})=7$ and $\dim(\gg_{\bar 1})=3$. If $\gg$ is naturally graded then the law of $\gg$  is isomorphic  to the following law

$$\begin{array}{ll}
L^{6,3} + \overline{\varphi}_{3,6}: \left\{\begin{array}{l}
[X_0,X_i]=X_{i+1}, \ 1 \leq i \leq 5, \\[1mm]
[X_0,Y_i]=Y_{i+1}, \ 1 \leq i \leq 2, \\[1mm]
(Y_1,Y_1)=6X_2,    \\[1mm]
(Y_1,Y_2)=3X_3,
\end{array}\right.&\begin{array}{ll}
(Y_1,Y_3)=X_4    \\[1mm]
(Y_2,Y_2)=2X_4    \\[1mm]
(Y_2,Y_3)=X_5    \\[1mm]
(Y_3,Y_3)=X_6
\end{array}  \end{array}$$
\end{thm}
\begin{proof}
Let $\gg = \gg_{\bar 0} \oplus \gg_{\bar 1}$ be a naturally graded filiform Lie superalgebra with $\dim(\gg_{\bar 0})=5$ and $\dim(\gg_{\bar 1}) = 3.$ Therefore $\gg_{\bar 0}=L_6$ and $gr(\gg)$ is

$$\begin{array}{ccccccccccc}
\underbrace{<X_0,X_1,Y_1>} & \oplus & \underbrace{<X_2, Y_2> }& \oplus & \underbrace{<X_3,Y_3> }& \oplus & \underbrace{<X_4>}& \oplus & \underbrace{<X_5>} & \oplus & \underbrace{<X_6>} \\
\gg^{1} & &  \gg^{2} & &  \gg^{3} & & \gg^{4} & & \gg^{5} & & \gg^{6}
\end{array}$$

\noindent As in previous cases, since the products we get $L^{6,3} + \Psi,$ with $\Psi = c\overline{\varphi}_{3,6} + a_1\Psi_{1,1}^2 + a_2 \Psi_{2,1}^3,$ where $\overline{\varphi}_{3,6} = 6\varphi_{1,2} + 2\varphi_{2,4} + \varphi_{3,6}$, and we have these products:

$$\begin{array}{ll}
L^{6,3}+\Psi: \left\{\begin{array}{l}
[X_0,X_i]=X_{i+1}, \ 1 \leq i \leq 5,  \\[1mm]
[X_0,Y_i]=Y_{i+1}, \ 1 \leq i \leq 2,  \\[1mm]
[X_1,Y_1]=a_1Y_2,         \\[1mm]
[X_1,Y_2]=(a_1-a_2)Y_3,	  \\[1mm]
[X_2,Y_1]=a_2Y_3,    \\[1mm]
(Y_1,Y_1)=6cX_2,
\end{array}\right.&\begin{array}{ll}
(Y_1,Y_2)=3cX_3    \\[1mm]
(Y_1,Y_3)=cX_4     \\[1mm]
(Y_2,Y_2)=2cX_4    \\[1mm]
(Y_2,Y_3)=cX_5     \\[1mm]
(Y_3,Y_3)=cX_6
\end{array}  \end{array} $$

\noindent with $c$ non-zero to avoid the trivial case. The condition $\Psi \circ \Psi(Y_1,Y_1,Y_1) = 0$ give us $a_2 = 0$, and since $\Psi \circ \Psi(X_1,Y_1,Y_2) = 0$ we have $a_1 = 0.$  By using a change  of scale $c = 1$ we obtain $L^{6,3} + \overline{\varphi}_{3,6}$.
\end{proof}

%%%%%%%%%%%%%%%%%%%%%%%%%%%%%%%%%%%%%%%%%%%%%%%%%%%%%%%%%%%%%%%
%%%%%%%%%%%%%%%%%%%%%%%%%%%%%%%%%%%%%%%%%%%%%%%%%%%%%%%%%%%%%%%
\section{Naturally Graded  (non-Lie) Leibniz Superalgebras with maximal super-nilindex}
%%%%%%%%%%%%%%%%%%%%%%%%%%%%%%%%%%%%%%%%%%%%%%%%%%%%%%%%%%%%%%%
%%%%%%%%%%%%%%%%%%%%%%%%%%%%%%%%%%%%%%%%%%%%%%%%%%%%%%%%%%%%%%%

Throughout this section we are going to study the Leibniz superalgebras with super-nilindex  $(n,m)$  because of the fact that in this case the even part $L_{\bar 0}$ is a null-filiform Leibniz algebra and the odd part $L_{\bar 1}$ has structure of filiform $L_{\bar 0}-$module. Consequently, it seems to be the first case to consider the ``naturally graded" structure. Note that all these Leibniz superalgebras with super-nilindex  $(n,m)$ are non-Lie ones and  contain in particular the only type of Leibniz superalgebra single generated or so-called null-filiform Leibniz superalgebras (for more details see \cite{null-filiform}). It is not difficult to see that these last superalgebras, in the case of non-degenerated,  are not naturally graded.

\

Next, we are going to recall the general expression of the Leibniz superalgebras with super-nilindex  $(n,m)$ (see \cite{perdido}). 

\begin{thm}\label{adapted1}
If $L=L_{\bar 0}\oplus L_{\bar 1}$ is a Leibniz superalgebra of super-nilindex  $(n,m)$, \black then there exists an adapted basis of $L$, namely $\{X_1,X_2, \dots , X_n,
Y_1, Y_2, \dots, Y_m\}$, with $\{X_1,X_2, \dots , X_n\}$ a basis of $L_{\bar 0}$ and $\{Y_1, Y_2, \dots , Y_m\}$ a basis of $L_{\bar 1}$, such that
$$\begin{array}{llc}
[X_i,X_1]=X_{i+1},& 1 \leq i \leq n-1, & [X_{n},X_1]=0\\[1mm]
[Y_j,X_1]=Y_{j+1},& 1 \leq j \leq m-1,& [Y_m,X_1]=0 \\[1mm]
\end{array}$$
Moreover, $[Y_j,X_k]=0$ for $1 \leq j \leq m$ and  $2 \leq k \leq n$, and the omitted products of $L_{\bar 0}=<X_1,X_2, \dots, X_n>$ vanish.
\end{thm}

Thus, it can be seen that if $L=L_{\bar 0} \oplus L_{\bar 1}$ is a Leibniz superalgebra of s-nilindex  $(n,m)$,  then its graded superalgebra associated $gr(L)=  \sum_{i \in \mathbb{N}} L^{i}$,  satisfying $[L^i,L^j]\subset L^{i+j}$, is exactly:

$$\begin{array}{ccccccc}
\underbrace{<X_1,Y_1>} & \oplus & \underbrace{<X_2, Y_2>}& \oplus & \underbrace{<X_3,Y_3>} & \oplus & \dots  \\
L^1 & & L^2 & & L^3 & &  
\end{array}$$

\noindent where for simplicity $L^1=L^1_{\bar 0} \oplus L^1_{\bar 1}=<X_1>\oplus<Y_1>$ has been replaced by $<X_1,Y_1>$ as there can not be any possible confusion between even elements, called $X_i$, and odd ones, called $Y_j$. Analogously, the same replacement has been applied to every $L^i$. The last terms of $gr(L)$ depend on three possibilities:

\

1. If $n<m$, then 
\begin{center}
$\begin{array}{ccccccccc}
\dots & \oplus & \underbrace{<X_{n},Y_n>} & \oplus & \underbrace{<Y_{n+1}>} & \oplus & \dots & \oplus & \underbrace{<Y_{m}>} \\
& &  L^{n}& &  L^{n+1}& &&&L^{m} \\
\end{array}$
\end{center}

2. If $n=m$, then 
\begin{center}
$\begin{array}{ccccc}
\dots & \oplus & \underbrace{<X_{n-1},Y_{n-1}>} & \oplus & \underbrace{<X_{n},Y_{n}>}  \\
& &  L^{n-1}& &  L^{n} \\
\end{array}$
\end{center} 

3. If $n>m$, then 
\begin{center}
$\begin{array}{ccccccccc}
\dots & \oplus & \underbrace{<X_{m-1},Y_m>} & \oplus & \underbrace{<X_{m}>} & \oplus & \dots & \oplus & \underbrace{<X_{n}>} \\
& &  L^{m}& &  L^{m+1}& &&&L^{n} \\
\end{array}$
\end{center}

\begin{rem}
It can be seen that the even elements of $gr(L)$, that is  $\sum_{i \in \mathbb{N}} L^i_{\bar 0}$,  constitutes a naturally graded null-filiform Leibniz algebra. That is,
$$NF^n : [X_i,X_1]=X_{i+1}, \quad 1 \leq i \leq n-1,$$
Furthermore, the definition of naturally graded Leibniz superalgebras is a generalization of naturally graded Leibniz algebras. Thus, if $L=L_{\bar 0} \oplus L_{\bar 1}$ is a naturally graded Leibniz superalgebra, then $L_{\bar 0}$ is a naturally graded Leibniz algebra. Indeed, $L^i_{\bar 0}$ is equivalent to ${\cal C}^{i-1}(L_{\bar 0})/{\cal C}^i(L_{\bar 0})$. Recall that the latter is used in the definition of naturally graded Leibniz algebras. 
\end{rem}

%%%%%%%%%%%%%%%%%%%%%%%%%%%%%%%%%%%%%%%%%%%%%%%%%%%%%%%%%%%%%%
%%%%%%%%%%%%%%%%%%%%%%%%%%%%%%%%%%%%%%%%%%%%%%%%%%%%%%%%%%%%%%
\subsection{Low dimensions}
%%%%%%%%%%%%%%%%%%%%%%%%%%%%%%%%%%%%%%%%%%%%%%%%%%%%%%%%%%%%%%
%%%%%%%%%%%%%%%%%%%%%%%%%%%%%%%%%%%%%%%%%%%%%%%%%%%%%%%%%%%%%%

Next we are going to obtain all the naturally graded  Leibniz superalgebras with super-nilindex  $(n,m)$  for low dimensions of either the even part or the odd part. We can accomplish such classification due to the results of \cite{perdido}. 

\begin{thm}
Let $L=L_{\bar 0} \oplus L_{\bar 1}$ be a non-degenerated naturally graded Leibniz superalgebras with s-nilindex  $(n,m)$,  then the law of $L$  is isomorphic to  a law  of the following list  where the omitted products are equal to zero.

\

\noindent {\bf List of  laws}. 

\

\noindent {\bf Pair of dimensions:} $n=2$, $m=2$. There is a  one-parametric family of naturally graded  Leibniz superalgebras of s-nilindex  $(2,2)$  whose law can be expressed in an adapted basis $\{X_1, X_2, Y_1, Y_2 \}$ as
$$\mu_{1}^{\a} =\left\{\begin{array}{lll}
[X_1,X_1]=X_2, &[X_1,Y_1]=\a Y_2,& \a \in \CC \\[1mm]
[Y_1,X_1]=Y_2, & [Y_1,Y_1]=X_2 &
\end{array}\right.$$

\noindent {\bf Pair of dimensions:} $n$ arbitray with $n \geq 3$, $m=2$. The only sub-case in which there are naturally graded superalgebras is exactly $n=3$. In this case there exist two non-isomorphic naturally graded  Leibniz superalgebras of s-nilindex  $(3,2)$  whose laws can be expressed in an adapted basis $\{X_1, X_2, X_3, Y_1, Y_2\}$ by:
$$\begin{array}{ll}
\mu_2 =\left\{\begin{array}{ll}
[X_i,X_1]=X_{i+1},& 1\leq i \leq 2\\[1mm]
[Y_1,X_1]=Y_2 & \\ [1mm]
[Y_1,Y_1]=X_2 & \\ [1mm]
[Y_2,Y_1]=X_3 &
\end{array}\right.&                    
\mu_3 =\left\{\begin{array}{ll}
[X_i,X_1]=X_{i+1}, 1\leq i \leq 2, & [Y_1,X_1]=Y_2 \\[1mm]
[X_1,Y_1]=-Y_2, & [Y_1,Y_2]=X_3 \\[1mm]
[Y_1,Y_1]=X_2,
\end{array}\right.  \end{array} $$ 

\

\noindent {\bf Pair of dimensions:} $n=2$, $m=3$. There are two non-isomorphic naturally graded  Leibniz superalgebras of s-nilindex  $(2,3)$  whose laws can be expressed in an adapted basis $\{X_1, X_2, Y_1, Y_2, Y_3\}$ by:
 $$\begin{array}{ll}\mu_{1} =\left\{\begin{array}{ll}
[X_1,X_1]=X_{2},& \\[1mm]
[Y_1,X_1]=Y_2 & \\[1mm]
[Y_2,X_1]=Y_3 & \\[1mm]
[Y_1,Y_1]=X_2 &
\end{array}\right.&
\mu_{3}=\left\{\begin{array}{ll}
[X_1,X_1]=X_2,& [Y_1,X_1]=Y_2 \\[1mm]
[X_1,Y_1]=-Y_2,& [Y_2,X_1]=Y_3 \\[1mm]
[X_1,Y_2]=-Y_3,& [Y_1,Y_1]=X_2
\end{array}\right.
\end{array}$$ 
\

\noindent {\bf Pair of dimensions:} $n=3$, $m=3$. There are two non-isomorphic naturally graded  Leibniz superalgebras of s-nilindex  $(3,3)$  whose laws can be expressed in an adapted basis $\{X_1, X_2, X_3, Y_1, Y_2, Y_3\}$ by:
$$\begin{array}{ll}
\mu_1 =\left\{\begin{array}{ll}
[X_i,X_1]=X_{i+1},& 1\leq i \leq 2\\[1mm]
[Y_j,X_1]=Y_{j+1}, & 1\leq j \leq 2\\[1mm]
[Y_1,Y_2]=X_3 & \\[1mm]
[Y_2,Y_1]=-X_3 & \\[1mm]
\end{array}\right.&
\mu_{8}=\left\{\begin{array}{ll}
[X_i,X_1]=X_{i+1},& 1\leq i \leq 2\\[1mm]
[Y_j,X_1]=Y_{j+1}, & 1\leq j \leq 2\\[1mm]
[X_1,Y_2]=-Y_3& \\[1mm]
[Y_1,Y_1]=X_2 & \\[1mm]
[Y_1,Y_2]=X_3 & \\[1mm]
\end{array}\right.
\end{array}$$

\

\noindent {\bf Pair of dimensions:} $n=4$, $m=3$. There are two non-isomorphic naturally graded Leibniz superalgebras of s-nilindex  $(4,3)$  whose laws can be expressed in an adapted basis $\{X_1, X_2, X_3, X_4, Y_1, Y_2, Y_3 \}$ by:
 $$\begin{array}{ll}\mu_9=\left\{\begin{array}{ll}
[X_i,X_1]=X_{i+1}, 1\leq i \leq 3, & [Y_1,Y_3]=X_4 \\[1mm]
[Y_j,X_1]=Y_{j+1}, 1\leq j \leq 2, & [Y_2,Y_2]=-X_4 \\[1mm]
[X_1,Y_1]=-Y_2, & [Y_3,Y_1]=X_4 \\[1mm]
[X_1,Y_2]=-Y_3 &
\end{array}\right.&    \mu_{12}=\left\{\begin{array}{ll}
[X_i,X_1]=X_{i+1},& 1\leq i \leq 3\\[1mm]
[Y_j,X_1]=Y_{j+1},& 1\leq j \leq 2\\[1mm]
[Y_1,Y_1]=X_2 & \\[1mm]
[Y_2,Y_1]=X_3 & \\[1mm]
[Y_3,Y_1]=X_4&
\end{array}\right.
\end{array}$$   

\

\noindent {\bf Pair of dimensions:} $n$ arbitrary with $n\geq 5$, $m=3$. There is  not any naturally graded  Leibniz superalgebra of s-nilindex  $(n,3)$. 

\

\noindent {\bf Pair of dimensions:} $n=2$, $m$ arbitrary with $m\geq 4$. There are two non-isomorphic naturally graded  Leibniz superalgebras of s-nilindex  $(2, m)$  whose laws can be expressed in an adapted basis $\{X_1,X_2,Y_1,Y_2,Y_3,Y_4,\dots,Y_m\}$ by
$$\begin{array}{ll} \mu_{m-1} =\left\{\begin{array}{ll}
[X_1,X_1]=X_{2}& \\[1mm]
[Y_i,X_1]=Y_{i+1}, & 1\leq i \leq m-1\\[1mm]
[Y_1,Y_1]=X_2 & \\[1mm]
\end{array}\right.& 
\mu_{m+1}=\left\{\begin{array}{ll}
[X_1,X_1]=X_{2}& \\[1mm]
[X_1,Y_j]=-Y_{j+1}, & 2\leq j \leq m-1\\[1mm]
[Y_i,X_1]=Y_{i+1}, & 1\leq i \leq m-1\\[1mm]
[Y_1,Y_1]=X_2 & \\[1mm]
\end{array}\right.
\end{array}$$
\end{thm}

%%%%%%%%%%%%%%%%%%%%%%%%%%%%%%%%%%%%%%%%%%%%%%%%%%%%%%%%%%%%%%%
%%%%%%%%%%%%%%%%%%%%%%%%%%%%%%%%%%%%%%%%%%%%%%%%%%%%%%%%%%%%%%%
\subsection{Cases $n \geq 3$ and $m \geq 4$}
%%%%%%%%%%%%%%%%%%%%%%%%%%%%%%%%%%%%%%%%%%%%%%%%%%%%%%%%%%%%%%%
%%%%%%%%%%%%%%%%%%%%%%%%%%%%%%%%%%%%%%%%%%%%%%%%%%%%%%%%%%%%%%%

Throughout this  sub-section  we deal with the problem of determining the structure of the aforementioned naturally graded Leibniz superalgebras for the remaining dimensions. Notice that the  naturally graded Leibniz superalgebra with maximal s-nilindex named $NG^{n, m}$ appears in every pair of arbitrary dimensions $n$ and $m$. 
$$NG^{n,m}:\left\{\begin{array}{ll}
[X_i,X_1]=X_{i+1},& 1\leq i\leq n-1 \\[1mm]
[Y_j,X_1]=Y_{j+1},& 1\leq j\leq m-1 \\[1mm]
[Y_i,Y_1]=X_{i+1},& 1\leq i\leq \min\{n-1,m\} 
\end{array}\right.$$
with the vectors $\{X_1,X_2,\dots,X_n,Y_1,Y_2,\dots,Y_m\}$ as adapted basis.  Also,  we have the following result.

\begin{thm}
Let $L=L_{\bar 0} \oplus L_{\bar 1}$ be a non-degenerated naturally graded Leibniz superalgebras with s-nilindex  $(n,m)$  and adapted basis $\{X_1,X_2,\dots,X_n,Y_1,Y_2,\dots,Y_m\}$.  Then the  law of $L$ will be isomorphic either to the law named $NG^{n,m}$ or to a law belonging to the following families of laws. The precise expression of these last families of laws depends on two possibilities, that is:
\begin{itemize}
\item[1.] If $n\leq m$, 
$$\left\{\begin{array}{ll}
 [X_i,X_1]=X_{i+1},& 1\leq i\leq n-1,\\[1mm]
 [Y_j,X_1]=Y_{j+1},& 1\leq j\leq m-1,\\[1mm]
 [X_1,Y_j]=-Y_{j+1},& 1\leq j\leq m-1,\\[1mm]
 [Y_1,Y_1]=\gamma_1 X_2,&\\[1mm]
 [Y_i,Y_1]=\gamma_i X_{i+1},& 3\leq i\leq n-1,\\[1mm]
 [Y_i,Y_j]=\displaystyle\sum_{s=0}^{j-1}(-1)^s \left ( \begin{array}{c}
 j-1\\
 s
 \end{array}\right ) \gamma_{i+s}X_{i+j},& 1\leq i\leq n-2,\ 2\leq j\leq n-i -1. \end{array}\right.$$
with $\gamma_2=0$ and $\gamma_{2j}=(-1)^{j}\displaystyle\sum_{s=0}^{j-2} (-1)^s \left ( \begin{array}{c}
j-1\\
s
\end{array}\right ) \gamma_{j+s+1},\ \mbox{ for } 2\leq j\leq \lfloor\frac{n-1}{2}\rfloor.$

\item[2.] If $n>m$,

$$\left\{\begin{array}{ll}
[X_i,X_1]=X_{i+1},& 1\leq i\leq n-1,\\[1mm]
[Y_j,X_1]=Y_{j+1},& 1\leq j\leq m-1,\\[1mm]
[X_1,Y_j]=-Y_{j+1},& 1\leq j\leq m-1,\\[1mm]
[Y_1,Y_1]=\gamma_1 X_2,&\\[1mm]
[Y_i,Y_1]=\gamma_i X_{i+1},& 3\leq i\leq m,\\[1mm]
[Y_i,Y_j]=\displaystyle\sum_{s=0}^{min\{j-1,m-i\}}(-1)^s \left ( \begin{array}{c}
j-1\\
s
\end{array}\right ) \gamma_{i+s}X_{i+j},& 2\leq j\leq m,\ 1\leq i\leq min\{m,n-j\}.
\end{array}\right.$$
with $\gamma_2=0$ and 
$$\gamma_{2j}=(-1)^{j}\sum_{s=0}^{j-2} (-1)^s \left ( \begin{array}{c}
j-1\\
s
\end{array}\right ) \gamma_{j+s+1},\ 2\leq j\leq \lfloor\frac{m}{2}\rfloor\qquad \qquad \gamma_i=\sum_{s=1}^{m-i}(-1)^{s+1} \left ( \begin{array}{c}
m\\
s
\end{array}\right )\gamma_{i+s}.$$
\end{itemize}
$\lfloor \ \rfloor$ denotes the whole part of the number.
\end{thm}

\begin{proof}
An adequate use of {\em Mathematica} software has been truly helpful to accomplish the proof of this theorem. It has been a very useful tool to both complete all the calculations involved in many particular pair of dimensions  and deduce the process to follow in arbitrary dimensions. Next, we will distinguish two cases. We only prove case 1 in details and omit case 2 completely in order not to repeat the same ideas.

\noindent {\bf Case 1: } $n\leq m$.
We consider the  adapted basis
$\{X_1,X_2,\dots,X_n,Y_1,Y_2,\dots,Y_m\}$ and the  gradation
$$<X_1,Y_1>\oplus <X_2,Y_2>\oplus\dots <X_n,Y_n>\oplus <Y_{n+1}>\oplus <Y_{n+2}>\oplus\dots\oplus <Y_m>$$

Using  their properties
$$\begin{array}{llll}
[X_i,X_1]=X_{i+1},& 1\leq i\leq n-1,&
[Y_j,X_1]=Y_{j+1},& 1\leq j\leq m-1,\\[1mm]
[X_1,Y_j]=\alpha_j Y_{j+1},& 1\leq j\leq m-1,&
[X_i,Y_j]=\beta_{i\,j}Y_{i+j},& 2\leq i\leq n,\quad 1\leq j \leq m-i\\[1mm]
[Y_i,Y_j]=\gamma_{i\,j}X_{i+j},& 1\leq i\leq m,\quad 1\leq j \leq n-i.&
\end{array}$$

We have that $[X_1,Y_1]+[Y_1,X_1]\in Ann(L),$ then $(\alpha_1+1)Y_2 \in Ann(L).$

{\bf Case 1.1: } $1+\alpha_1\neq 0$, then $Y_2 \in Ann(L).$ Taking into account that $[Y_2,X_1]=Y_3$, it is easy to prove that $Y_3\in Ann(L).$ Analogously, $Y_4,Y_5,\dots,Y_m\in Ann(L).$ Thus, we have
$$\begin{array}{llll}
[X_i,X_1]=X_{i+1},& 1\leq i\leq n-1,&
[Y_j,X_1]=Y_{j+1},& 1\leq j\leq m-1,\\[1mm]
[X_1,Y_1]=\beta_1 Y_2,& &
[X_i,Y_1]=\beta_i Y_{i+1},& 2\leq i\leq n,\\[1mm]
[Y_i,Y_1]=\gamma_i X_{i+1},& 1\leq i\leq n-1.&&
\end{array}$$

We make the super Leibniz identity on the triples $\{Y_i,Y_1,X_1\}$ and we obtain $\gamma_i=\gamma_1=\gamma$ with $1\leq i\leq n-1.$ If $\gamma=0$ then, the superalgebra is degenerated, thus $\gamma\neq 0$ and making $Y'_1=\frac{1}{\sqrt{\gamma}}Y_1$ allows us to consider $\gamma=1.$

Now, from the super Leibniz identity on the triples $\{X_i,Y_1,Y_1\}$ for $1\leq i\leq n-2$ we get $\beta_i=0.$ From $\{X_{n-2},Y_1,X_1\}$ and $\{X_{n-1},Y_1,X_1\}$ we obtain $\beta_{n-1}=0$ and  $\beta_{n}=0,$ respectively.

Thus, we obtain the naturally graded  Leibniz superalgebra  $NG^{n,m}$.

\

{\bf Case 1.2: } $1+\alpha_1=0.$ We can distinguish two cases:

\

{\bf (a)} $1+\alpha_i=0$ for $2\leq i\leq m-1.$ We have:
$$\begin{array}{llll}
[X_i,X_1]=X_{i+1},& 1\leq i\leq n-1,&
[Y_j,X_1]=Y_{j+1},& 1\leq j\leq m-1,\\[1mm]
[X_1,Y_j]=-Y_{j+1},& 1\leq j\leq m-1,&
[X_i,Y_j]=\beta_{i\,j}Y_{i+j},& 2\leq i\leq n,\quad 1\leq j \leq m-i\\[1mm]
[Y_i,Y_j]=\gamma_{i\,j}X_{i+j},& 1\leq i\leq m,\quad 1\leq j \leq n-i.&&
\end{array}$$

The super Leibniz identity on the following triples imposes further constraints on the above family.
\begin{center}
\begin{tabular}{lll}
	Super Leibniz identity& &Constraint\\
		\hline\hline\\
		$\{X_1,X_1,Y_j,\},\ 1\leq j\leq m-2$ &\quad $\Rightarrow $\quad & $\beta_{2\, j}=0$\\
		$\{X_1,X_i,Y_j\},\ 3\leq i\leq n,\ 1\leq j\leq m-i-1$ &\quad $\Rightarrow $\quad& $\beta_{i\, j}=0$  \mbox{and}  \\
		&  & $[X_i,Y_{m-i}]=\beta_{i\,m-i}Y_{m}=\beta_i Y_m,\ \ 3\leq i\leq n$\\

$\{X_i,X_1,Y_{m-i-1},\},\ 2\leq i\leq n-1$ &\quad $\Rightarrow $\quad & $\beta_3=0$ and $\beta_{i+1}=(-1)^{i+1}\beta_3=0$\\
		
$\{Y_i,X_1,Y_{j-1}\},\ 1\leq i\leq n-2,\ 2\leq j\leq n-i$ &\quad $\Rightarrow $\quad&$\gamma_{i\, j}=\gamma_{i\, j-1}-\gamma_{i+1\,  j-1}$ \\
\end{tabular}
\end{center}

Thus, we prove by the induction method that:
$$\gamma_{i\, 1}=\gamma_i,\ 1\leq i\leq n-1,\qquad 
\gamma_{i\, j}=\displaystyle\sum_{ k=0}^{j-1}(-1)^{ k } \left ( \begin{array}{c}
j-1\\
k
\end{array}\right ) \gamma_{i+k},\ \ 1\leq i\leq n-2,\ 2\leq j\leq n-i$$

Applying the above relations, we can write 
$$\begin{array}{llll}
[X_i,X_1]=X_{i+1},& 1\leq i\leq n-1,&
[Y_j,X_1]=Y_{j+1},& 1\leq j\leq m-1 \\[1mm]
[X_1,Y_j]=-Y_{j+1},& 1\leq j\leq m-1,&
[Y_i,Y_1]=\gamma_i X_{i+1},& 1\leq i\leq n-1 \\[1mm]
[Y_i,Y_j]=\displaystyle\sum_{k=0}^{j-1}(-1)^k \left ( \begin{array}{c}
j-1\\
k
\end{array}\right ) \gamma_{i+k}X_{i+j},& 1\leq i\leq n-2, & 2\leq j\leq n-i \end{array}$$

For instance, the super Leibniz identity on the triples $[X_1,[Y_1,Y_1]]$ and $[X_1,[Y_j,Y_j]]$  gives respectively  $\gamma_2=0$ and
$$\gamma_{2j}=(-1)^{j}\displaystyle\sum_{s=0}^{j-2} (-1)^s \left ( \begin{array}{c}
j-1\\
s
\end{array}\right ) \gamma_{j+s+1},\ 2\leq j\leq \lfloor\frac{n-1}{2}\rfloor.$$

We obtain the family of naturally graded Leibniz superalgebras with maximal s-nilindex of the theorem statement.
 
\

{\bf (b)} There exists $k$ such that $1+\alpha_i=0$ for $2\leq i\leq k-2$, and $1+\alpha_{k-1}\neq 0.$ We have
$$[X_1,Y_{k-1}]+[Y_{k-1},X_1]=(1+\alpha_{k-1})Y_k\in Ann(L).$$ Then $Y_k\in Ann(L) \mbox{ and } Y_{k+1},Y_{k+2},\dots,Y_m\in Ann(L).$ Hence, we can consider
$$\begin{array}{llll}
[X_i,X_1]=X_{i+1},& 1\leq i\leq n-1,&
[Y_j,X_1]=Y_{j+1},& 1\leq j\leq m-1,\\[1mm]
[X_1,Y_j]=- Y_{j+1},& 1\leq j\leq k-2,&
[X_1,Y_{k-1}]=\alpha Y_k,& 1+\alpha\neq 0\\[1mm]
[X_i,Y_j]=\beta_{i\, j}Y_{i+j},& 1\leq j \leq k-1,\quad 2\leq i\leq min\{n,m-j\}& &\\[1mm]
[Y_i,Y_j]=\gamma_{i\, j}X_{i+j},& 1\leq j\leq k-1,\quad 1\leq i \leq n-k+1.& &   \end{array}$$

%\noindent $\bullet$ We suppose that $k\leq n.$

Consider the super Leibniz identity on the triple $\{Y_i,X_1,Y_j\}$ with $2\leq j\leq k-1$ and $1\leq i\leq n-j$. We get $$\gamma_{i\, j}=\gamma_{i\, j-1}-\gamma_{i+1\, j-1}.$$

For $j=k,$ we have that $\gamma_{i\, k-1}=\gamma_{1\, k-1}.$
If we rename  $\gamma_{i\, 1}:=\gamma_i$  with $1\leq i\leq n-1$,  it is easy to prove by induction method on $j,$ that
\begin{equation}\label{estrella}
\gamma_{i\, j}=\displaystyle\sum_{s=0}^{j-1}(-1)^{s}\left ( \begin{array}{c}
j-1\\
s
\end{array}\right )\gamma_{i+s},\quad\quad 2\leq j\leq k-1,\ \ 1\leq i\leq n-j
\end{equation}

Thus, we have that
$$\begin{array}{ll}
[Y_i,Y_1]=\gamma_{i}X_{i+1},& 1\leq i\leq n-1 \\[1mm]
[Y_i,Y_j]=\displaystyle\sum_{s=0}^{j-1}(-1)^{s}\left ( \begin{array}{c}
j-1\\
s
\end{array}\right )\gamma_{i+s}X_{i+j},& 2\leq j\leq k-1,\ \ 1\leq i\leq n-j
\end{array}$$

From the super Leibniz identity on the triple $\{X_1,Y_1,Y_1\}$ we get $\gamma_2=0.$

%From $[X_1,[Y_j,Y_j]]$ we get $$\gamma_{2j}=(-1)^{j}\displaystyle\sum_{s=0}^{j-2}(-1)^{s}\left ( \begin{array}{c}
%j-1\\ s \end{array}\right )\gamma_{j+s+1}$$ with $2\leq j\leq \lfloor\frac{n-1}{2}\rfloor.$

Using the induction method together with super Leibniz identity on the triple $\{X_i,X_1,Y_j\}$, we prove that
$$\beta_{i\, j}=\left\{\begin{array}{ll}
0, & 3\leq i+j\leq k-1 \\
(-1)^{j-k+1} \left ( \left (\begin{array}{c}
i-2\\
j-k+i
\end{array}\right )+\left (\begin{array}{c}
i-1\\
j-k+i
\end{array}\right )\alpha \right ), & k\leq i+j\leq m
\end{array}\right.
$$
with $2\leq i\leq \min\{n,m-j\}.$

\

\noindent $(1)$ $\bf m\geq n+3.$
If we make the super Leibniz identity on the triples 
$\{X_n,X_1,Y_1\}$ and $\{X_n,X_1,Y_2\}$ we get $\alpha=-\displaystyle\frac{k-2}{n}$ and $\alpha=-\displaystyle\frac{k-3}{n},$ respectively. But that is a contradiction, then there is not any superalgebra.

\

\noindent $(2)$ $\bf m=n+2.$

\begin{center}
	\begin{tabular}{lll}
		Super Leibniz identity & & Constraint\\
		\hline\hline\\
		$\{X_n,X_1,Y_{1},\}$ & $\Rightarrow $ &$\alpha=-\displaystyle\frac{k-2}{n}.$\\[1mm]
		$\{Y_i,Y_1,Y_{n+1-i}\}$ & $\Rightarrow $\quad&$
		\left\{\begin{array}{l}
		\gamma_i \beta_{i+1\, n+1-i}=0,\\
		\beta_{i+1\, n+1-i}\neq 0,\\ i\geq n+2-k
		\end{array}\right.$ $\Rightarrow$
	  $\gamma_i=0,\ n+2-k\leq i\leq n-1.$
		 \\[5mm]
		$\{Y_i,Y_1,Y_1\}$& $\Rightarrow $\quad&$
		\left\{\begin{array}{l}
		\gamma_i\beta_{i+1\,1}=0,\\
		\beta_{i+1\,1}\neq 0,\\ k-2\leq i\leq  n
		\end{array}\right.$ $\Rightarrow$
		$\gamma_i=0,\ k-2\leq i\leq n-1.$
		\\[5mm]
		$\{X_1,Y_{k-1},Y_j\}$& $\Rightarrow $\quad&
		$\alpha \gamma_{k\, j}-\gamma_{j+1\, k-1}=0,$\ $1\leq j\leq n-k.$
	\end{tabular}
\end{center}

By other hand, recall that $\gamma_j=0$, for $k-2\leq j\leq n-1$, and by Equation \eqref{estrella} we have that $\gamma_{k\, j}=0.$ Accordingly $\gamma_{j+1\,k-1}=0$ for $1\leq j\leq n-k,$ that is,
$$\sum_{s=0}^{k-2} (-1)^s \left  (\begin{array}{c}
k-2\\
s
\end{array}\right )\gamma_{j+1+s}=0,\quad 1\leq j\leq n-k$$

Then we have a system of  $(n-k)$ linear equations and $(n-k)$ variables  admitting  one unique solution $\gamma_i=0$, $2\leq i\leq n-k+1.$

Only rest to prove that $\gamma_1=0.$ For this, it is sufficient to consider the super Leibniz identity on the triple $\{Y_1,Y_1,Y_{k-1}\}.$ We conclude that the superalgebra is degenerated.

\
	
\noindent $(3)$ $\bf m=n+1.$
From the super Leibniz identity on $\{Y_i,Y_1,Y_1\}$ we get $\gamma_i\beta_{i+1\,1}=0$. If $k-2\leq i\leq  n$ then $\beta_{i+1\,1}\neq 0.$ Thus $\gamma_i=0$ for $k-2\leq i\leq n-1.$

Analogously, from $\{Y_{k-3},Y_1,Y_3\}, \{X_1,Y_1,Y_j\}$ we conclude that $\gamma_{k-3}=0$, and $\gamma_{2\, j}+\gamma_{j+1}=0,$ respectively. By other hand, remenber that $\gamma_j=0$ with $k-2\leq j\leq n-1$, which implies that
$$\sum_{s=1}^{j-2} (-1)^s  \left (\begin{array}{c}
j-1\\
s
\end{array}\right)\gamma_{2+s}+((-1)^{j-1}+1)\gamma_{j+1}=0,\quad 4\leq j\leq k-3$$

Then we have a system of  $(k-6)$ linear equations and $(k-6)$ variables admiting one unique solution $\gamma_i=0$ with $3\leq i\leq k-4.$

Only rest to prove that $\gamma_1=0.$ For this, it is sufficient to consider the super Leibniz identity on the triples $\{Y_1,Y_1,Y_{k-1}\}$ if $k$ is odd and   $\{Y_1,Y_1,Y_{k-2}\}$ if $k$ is  even. We conclude then, that the superalgebra is degenerated.
%\end{itemize}

\

\

%{\bf Case 2: } $n=m$.
%\textcolor{red}{Lo he hecho aparte y es igual que el caso $n<m$. Hay que añadir el igual.} {\blue This case is analogous to the previous case.}

{\bf Case 2: } $n>m$.
We consider the adapted basis $\{X_1,X_2,\dots,X_n,Y_1,Y_2,\dots,Y_m\}$ and the  gradation
$$<X_1,Y_1>\oplus <X_2,Y_2>\oplus\dots <X_m,Y_m>\oplus <X_{m+1}>\oplus <X_{m+2}>\oplus\dots\oplus <X_n>$$

 By using  the adapted basis and the properties of natural gradation
$$\begin{array}{llll}
[X_i,X_1]=X_{i+1},& 1\leq i\leq n-1,&
[Y_j,X_1]=Y_{j+1},& 1\leq j\leq m-1\\[1mm]
[X_1,Y_j]=\alpha_j Y_{j+1},& 1\leq j\leq m-1,&
[X_i,Y_j]=\beta_{i\,j}Y_{i+j},& 1\leq j \leq m-2,\quad 2\leq i\leq m-j\\[1mm]
[Y_i,Y_j]=\gamma_{i\,j}X_{i+j},& 1\leq i\leq m,& 1\leq j \leq \min\{ n-i, \ m\}
\end{array}$$

{\bf Case 2.1} If $1+\alpha_1\neq 0,$ then $Y_2,Y_3,\dots,Y_m\in Ann(L).$ We have
$$\begin{array}{llll}
[X_i,X_1]=X_{i+1},& 1\leq i\leq n-1,&
[Y_j,X_1]=Y_{j+1},& 1\leq j\leq m-1\\[1mm]
[X_1,Y_1]=\alpha Y_{2},& &
[X_i,Y_1]=\beta_{i}Y_{i+1},&  2\leq i\leq m-1\\[1mm]
[Y_i,Y_1]=\gamma_{i}X_{i+1},& 1\leq i\leq m
\end{array}$$

Similar to the previous cases, using the super Leibniz identity, we get  the superalgebra $NG^{n,m}.$

%From $[Y_i,[Y_1,X_1]]$ we have $\gamma_{i}=\gamma_1=\gamma$ with $1\leq i\leq {\red n-1} \ {\blue m}.$ If $\gamma=0$ then, the superalgebra is split, thus $\gamma\neq 0$ and making $Y'_1=\frac{1}{\sqrt{\gamma}}Y_1$ we allows $\gamma=1.$
%From $[X_1,[Y_1,Y_1]]$ we get to $\alpha=0.$
%From $[X_i,[Y_1,Y_1]]$ we have $\beta_i=0$ with $2\leq i\leq n-2.$
%Thus, we have the following superalgebra:
%$$\left\{\begin{array}{ll}
%[X_i,X_1]=X_{i+1},& 1\leq i\leq n-1,\\[1mm]
%[Y_j,X_1]=Y_{j+1},& 1\leq j\leq m-1,\\[1mm]
%[Y_i,Y_1]=X_{i+1},& 1\leq i\leq {\red n-1} \ {\blue m}.
%\end{array}\right.$$
%\noindent {\blue which corresponds with $NG^{n,m}$.}

\

{\bf Case 2.2: } $1+\alpha_1=0.$ We  distinguish  the following two cases:

{\bf (a)} $1+\alpha_i=0$ for all $i,\ 2\leq i\leq m-1.$ We have:
$$\begin{array}{llll}
[X_i,X_1]=X_{i+1},& 1\leq i\leq n-1,&
[Y_j,X_1]=Y_{j+1},& 1\leq j\leq m-1\\[1mm]
[X_1,Y_j]=-Y_{j+1},& 1\leq j\leq m-1,&
[X_i,Y_j]=\beta_{i\,j}Y_{i+j},& 1\leq j\leq m-2,\quad 2\leq i \leq m\\[1mm]
[Y_i,Y_j]=\gamma_{i\,j}X_{i+j},& 1\leq i\leq m,& 1\leq j \leq \min\{m,n-i\}
\end{array}$$

%From $[X_1,[X_1,Y_j]]$ with $1\leq j\leq m-2$ we have that $\beta_{2\, j}=0.$
%From $[X_1,[X_i,Y_j]]$ with $1\leq j\leq m-2 ,$ $2\leq i\leq m-1-j,$ we get  $\beta_{i\, j}=0.$
%Thus, we obtain:
%$$[X_i,Y_{m-i}]=\beta_{i\,m-i}Y_{m}=\beta_i Y_m,\ \ 3\leq i\leq n$$
%From $[X_i,[X_1,Y_{m-i-1}]]$ with $2\leq i \leq n-1,$ we get $\beta_3=0$ and $\beta_{i+1} =(-1)^{i+1}\beta_3=0.$
%From $[Y_i,[X_1,Y_{j-1}]]$ with $1\leq i\leq m-2,$ $2\leq j\leq n-i$, we get 
%$$\gamma_{i\, j}=\gamma_{i\, j-1}-\gamma_{i+1\,  j-1}$$
%Thus, we prove by induction that:
%$$\gamma_{i\, 1}=\gamma_i,\ 1\leq i\leq m,$$
%$$\gamma_{i\, j}=\sum_{ k=0}^{min\{j-1,m-i\}}(-1)^{  k} \left ( \begin{array}{c}
%j-1\\
%k
%\end{array}\right ) \gamma_{i+k},\ \ 1\leq i\leq n-2,\ 2\leq j\leq n-i.$$
%Thus, we have the family:
%$$\begin{array}{ll}
%[X_i,X_1]=X_{i+1},& 1\leq i\leq n-1,\\[1mm]
%[Y_j,X_1]=Y_{j+1},& 1\leq j\leq m-1,\\[1mm]
%[X_1,Y_j]=-Y_{j+1},& 1\leq j\leq m-1,\\[1mm]
%[Y_i,Y_1]=\gamma_i X_{i+1},& 1\leq i\leq m,\\[1mm]
%[Y_i,Y_j]=\displaystyle\sum_{k=0}^{j-1}(-1)^k \left ( \begin{array}{c}
%j-1\\
%k
%\end{array}\right ) \gamma_{i+k}X_{i+j},& 2\leq j\leq m,\ 1\leq i\leq min\{m,n-j\}.
%\end{array}$$
%From $[X_1,[Y_1,Y_1]]$ we have that $\gamma_2=0.$
%From $[X_1,[Y_j,Y_j]]$ we get 
%$$\gamma_{2j}=(-1)^{j}\sum_{s=0}^{j-2} (-1)^s \left ( \begin{array}{c}
%j-1\\
%s
%\end{array}\right ) \gamma_{j+s+1},\ 2\leq j\leq \lfloor\frac{m}{2}\rfloor.$$
%From $[Y_i,[Y_m,X_1]]$ with $1\leq i\leq n-m-1 $, we obtain
%$$\gamma_i=\sum_{s=1}^{m-i}(-1)^{s+1} \left ( \begin{array}{c}
%m\\
%s
%\end{array}\right )\gamma_{i+s}.$$

We obtain the family of naturally graded Leibniz superalgebras with maximal s-nilindex of the theorem statement.

\

{\bf (b)} There exists $k$, with $2\leq k\leq m-1,$ such that $1+\alpha_i=0$ for $2\leq i\leq k-2$ and $1+\alpha_{k-1}\neq 0.$ We have $$[X_1,Y_{k-1}]+[Y_{k-1},X_1]=(1+\alpha_{k-1})Y_k\in Ann(L).$$ Then $Y_k\in Ann(L) \mbox{ and } Y_{k+1},Y_{k+2},\dots,Y_m\in Ann(L).$
Thus,
$$\begin{array}{llll}
[X_i,X_1]=X_{i+1},& 1\leq i\leq n-1,&
[Y_j,X_1]=Y_{j+1},& 1\leq j\leq m-1\\[1mm]
[X_1,Y_j]=- Y_{j+1},& 1\leq j\leq k-2,&
[X_1,Y_{k-1}]=\alpha Y_k,& 1+\alpha\neq 0\\[1mm]
[X_i,Y_j]=\beta_{i\, j}Y_{i+j},& 1\leq j \leq k-1,\quad 2\leq i\leq m-j&\\[1mm]
[Y_i,Y_j]=\gamma_{i\, j}X_{i+j},& 1\leq j\leq k-1,\quad 1\leq i \leq min\{m,n-j\}& &
\end{array}$$

 By using  the same techinques as in the case $1$, we conclude that the obtained superalgebra is degenerated.

\end{proof}

\bibliographystyle{amsplain}

\end{document}